\documentclass[journal]{IEEEtran}
\ifCLASSINFOpdf
\else
\fi

\usepackage[pdftex]{graphicx}
\usepackage{epstopdf}
\usepackage{amsmath} 
\usepackage{amssymb} 
\usepackage{amsthm}
\usepackage{cite}
\usepackage{color}
\usepackage{tikz}
\usetikzlibrary{calc}

\usepackage[caption=false,font=footnotesize]{subfig}

\newtheorem{theorem}{Theorem}
\theoremstyle{definition}
\newtheorem{definition}{Definition}
\newtheorem{proposition}{Proposition}
\newtheorem{corollary}{Corollary}
\newtheorem{assumption}{Assumption}
\newtheorem{lemma}{Lemma}

\DeclareMathOperator{\sign}{sign}

\newcommand{\rPose}{g}
\newcommand{\pVar}{q}
\newcommand{\applyG}[2]{#1(#2)}

\newcommand{\rCoordX}{x}
\newcommand{\rCoordY}{y}
\newcommand{\rCoordZ}{z}

\newcommand{\rCoordR}{R}
\newcommand{\rP}{p_r}

\newcommand{\rCoordsSEThreeAt}[1]{(\rCoordX(#1),\rCoordY(#1),\rCoordZ(#1),\rCoordR(#1))}

\newcommand{\rCoordsSEThreeSub}[1]{(x_{#1}, y_{#1}, z_{#1}, R_{#1})}

\newcommand{\xCoordsAtSEThree}[1]{(\rCoordX(#1),\rCoordY(#1), \rCoordZ(#1))}
\newcommand{\rCoordRAt}[1]{R(#1)}

\newcommand{\vPose}{v}
\newcommand{\vCoordX}{v_x}
\newcommand{\vCoordY}{v_y}
\newcommand{\vCoordZ}{v_z}

\newcommand{\vCoordsSEThree}{(\vCoordX, \vCoordY, \vCoordZ)}

\newcommand{\oP}{p_o}

\newcommand{\sPose}{g_r}

\newcommand{\sCoordsSupSEThree}[1]{(x^{#1}_r, y^{#1}_r, z^{#1}_r, R^{#1}_r)}

\newcommand{\sCoordsSupSub}[2]%
  {(x^{#1}_{r,#2}, y^{#1}_{r,#2}, \theta^{#1}_{r,#2})}
\newcommand{\sCoordsSupSubSEThree}[2]%
  {(x^{#1}_{r,#2}, y^{#1}_{r,#2}, z^{#1}_{r,#2}, R_{r,#2})}

\newcommand{\halfL}{\sigma}

\newcommand{\Nrect}{N_r}

\newcommand{\halfLSq}{\sigma}

\newcommand{\roboF}{\mathcal{R}}
\newcommand{\roboPos}[1]{p\left(#1\right)}

\newcommand{\roboPosDer}[1]{\dot{p}\left(#1\right)}
\newcommand{\roboVel}{e_0}
\newcommand{\roboControlLVel}{u}
\newcommand{\roboControlAVel}{\omega}
\newcommand{\roboControlAVelAt}[1]{(\omega_1(#1),\omega_2(#1), \omega_3(#1))}
\newcommand{\roboCurvature}{\kappa}
\newcommand{\roboRadiusCurvature}{\mathcal{R}_\kappa}

\begin{document}
%
\title{Bendable Cuboid Robot Path Planning with Collision Avoidance using Generalized $L_p$ Norms}

%
%
%

\author{Nak-seung P.~Hyun$^{1}$,~\IEEEmembership{Student Member,~IEEE,}
        Patricio A.~Vela$^{2}$,~\IEEEmembership{Member,~IEEE,}
        and~Erik I.~Verriest$^{3}$,~\IEEEmembership{Fellow,~IEEE}
\thanks{The authors are with the Department
of Electrical and Computer Engineering, Georgia Institute of Technology, Atlanta,
GA, 30332 USA e-mail: (1. nhyun3@gatech.edu, 2. pvela@gatech.edu, 3. erik.verriest@ece.gatech.edu).}
\thanks{1. N.P. Hyun is the corresponding author}
}
\maketitle

\begin{abstract}
Optimal path planning problems for rigid and deformable (bendable) cuboid
robots are considered by providing an analytic safety constraint using
generalized $L_p$ norms.  For regular cuboid robots, level sets of weighted
$L_p$ norms generate implicit approximations of their surfaces.  For
bendable cuboid robots a weighted $L_p$ norm in polar coordinates implicitly
approximates the surface boundary through a specified level set.  Obstacle
volumes, in the environment to navigate within, are presumed to be
approximately described as sub-level sets of weighted $L_p$ norms.
Using these approximate surface models, the optimal safe path planning
problem is reformulated as a two stage optimization problem, where the
safety constraint depends on a point on the robot which is closest to the
obstacle in the obstacle's distance metric. A set of equality and inequality
constraints are derived to replace the closest point problem, which is then
defines additional analytic constraints on the original path planning problem. 
Combining all the analytic constraints with logical AND operations leads to a
general optimal safe path planning problem.  Numerically solving the problem
involve conversion to a nonlinear programing problem.  Simulations for rigid
and bendable cuboid robot verify the proposed method.

\end{abstract}

\begin{IEEEkeywords}
Optimization and optimal control, collision avoidance, motion and path planning, soft robotics
\end{IEEEkeywords}

%
\IEEEpeerreviewmaketitle

%
%

%
%

\section{Introduction}
\label{sec:intro}


\IEEEPARstart{P}{laning} an optimal safe path for a robot which minimizes control cost (e.g, path length) while avoiding collisions with obstacles in three dimensional space requires compatible kinematic and geometric models for tractably describing the collision constraints.
The kinematics or dynamics of the 3D rigid robot are analytically formulated with the special Eucledian group, $SE(3)$ \cite{murray1994mathematical}. For example, the rigid body group model applies to quadrocopter \cite{hehn2011quadrocopter} or bipedal robot control \cite{grizzle20103d}. 
The difficulty in finding an optimal path is due to the formulation of
the anti-collision configurations in $SE(3)$. A classical way to find
safe configurations is through the \emph{configuration space approach},
which transfers the rigid body path planning problem to a point mass
problem by enlarging the obstacle with the robot in $\mathbb{R}^3$ for
every orientation. The enlarged obstacles are denoted as C-obstacles,
and the set of collision-free points is denoted as C-free space, which
is the complement of the C-obstacle space in $SE(3)$. Optimal path
planning establishes trajectory feasibility by checking if the robot's
configuration lies in the C-free space for all time. 

The analytical expression of C-free space for polygon shaped obstacles contains a set of inequalities combined with logical OR operations \cite{lavalle2006planning}, which is analytically intractable since general optimization algorithms are based on logical AND relations \cite{Taha2006}. Therefore, most of the optimal safe path planning research includes either sampling the configuration space and discretely searching the feasible path in the C-free space (e.g., sampling based planning \cite{lavalle2006planning}) or converting the OR operations to AND operation by introducing integer values for the collision constraints and using mixed integer programing (MIP) algorithms \cite{schouwenaars2001mixed}. However, integer programming is known to be an NP-hard problem \cite{Karp2010}. Recently, a method to reduce the number of integers for cluttered obstacles has been introduced in \cite{deits2015efficient}, at the expense of reduced probability of finding a feasible path.  On the sampling side, \cite{Yan2015} derived a closed form Minkowski sum of two ellipsoids, which was later demonstrated to reduce the computational load of sampling based planning in \cite{Yan2016}.  Only ellipsoidal shapes were considered.

Inspired by the geometry of level sets of $L_p$ norms, our prior work \cite{Hyun2017Lp} proposed a new framework based on analytic inequalities combined with AND operations that represent the C-free space of a rectangular robot with rectangular or ellipsoidal obstacles in the case of planar $SE(2)$ models.  This new framework does not requires integer variables but is restricted to the planar model since \emph{edge to edge} collision in $\mathbb{R}^3$ cannot be captured.  This paper further generalizes the framework of weighted $L_p$ norms by analytically approximating the $SE(3)$ C-free space for regular rigid cuboid robots (which can be relaxed to ellipsoid or circular robots) with cuboid obstacles (which can also be relaxed to ellipsoid or spherical shapes). 

Continuing, there are many mechanical robots with non-rigid (deformable
or soft) bodies. These robots can change their body shape during
navigation or while executing a movement. For example, robotic fish bend
their bodies to maneuver under water
\cite{marchese2014autonomous,curet2011mechanical},
and a finger shaped manipulator using fiber-reinforced soft actuators in
\cite{Connolly03012017} bends, twists and enlarges its body. Bending
robots figure extensively as models for continuum manipulators
\cite{Li2014Continuum, Godage2011Continuum}, where a multi-segment
cylinder-shaped robot is controlled by a constant curvature bend in each
segment \cite{Kahrs2015ContinuumSurvey}.

However, the difficulty of applying rigid motion theory to the non-rigid motion directly is the lack of a meaningful joint space to describe the deformable motion \cite{Schultz2016SoftPath}. Therefore, most research in controlling deformable robots first approximates the deformable surfaces. Deformation of solid geometric objects has been studied extensively in computer graphics \cite{Moore2006survey}.  One popular non-physical model uses the free form deformation technique, where the surface deformations use a trivariate tensor product Berstein polynomial with a discrete set of control points in $\mathbb{R}^3$ describing the deformations. This method has been extensively used in computer graphics but the control points lack physical meaning for robotics applications. As a result, much path planning research on deformable robots uses a physically grounded deformation model acting on sampled points of the surface \cite{nealen2006physically}. The most accurate approximation of the surface, as far as we know, is by solving the partial differential equation (PDE) governing dynamic elastic equation using finite element methods (FEM) \cite{Rodriguez2006Deformable,gayle2005constraint}.  The surface is discretized by the meshes (elements), and the PDE is converted to the algebraic equations of the node points. The original continuous shape function of meshes are then approximated by basis functions using the node values. Despite the accuracy of the model, the computation load for solving FEM is very high. Alternatively, a mass-spring system is widely used to represent the surface by discretizing the surface with a finite set of point masses which are connected with spring \cite{anshelevich2000deformable, gayle2005path}. Benefits of having this model is that the governing dynamics equation become ordinary differential equations, in contrast to FEM where the planner needs to solve PDE. The challenge in path planning using the sample based deformable surface lies in the collision checking part since the planner needs to consider every possible collision on the discrete surfaces.  For deformable robots, since safety constraint cannot be expressed analytically, the planning research is mostly limited to sample based planning methods such as probabilistic roadmap (PRM) \cite{anshelevich2000deformable,gayle2005path} or rapidly-exploring random tree (RRT) \cite{Rodriguez2006Deformable} planners.

This paper includes an analytic approximation for the deformable surface of a bendable cuboid robot. It lies in the middle ground between physical and non-physical approximations since the control parameter of the deformable shape is the curvature of bending (which can be controlled in many applications, \cite{Li2014Continuum, Polygerinos2015Bending, Chang2017Snake}) but is not influenced by external forces. This analytic model allows the construction of a set of equality and inequality constraints combined with logical AND operations as a function of the configuration of the robot in $SE(3)$ and the curvature control parameter, which altogether represent the safety constraints for the bendable cuboid robot. The proposed bendable cuboid model can be used for the deformable safety guard of a continuum manipulator \cite{Li2014Continuum}, a soft fiber-reinforced bending actuator \cite{Polygerinos2015Bending}, and an average snake body frame \cite{Chang2017Snake}. 

\subsection{Contributions}
\subsubsection{Analytic surface model for bendable robot} A family of positive definite functions in $\mathbb{R}^2$ and $\mathbb{R}^3$ parameterized by some even number $p$ and the shape parameters of robot (half-lengths of cuboid and the curvature $\roboCurvature$) is introduced, whose $|\roboCurvature|$-level set represents the boundary of bent rectangular robot or bent cuboid robot, respectively. The formulation relies on a weighted $L_p$ norm in polar coordinates, which represents the exact boundary of bent cuboid robots as $p$ goes to infinity.  In addition, the length of the center line (to be precisely defined in \S\ref{sec:Model}) and the volume are preserved regardless of the bending. For the case of approximating a regular cuboid (no bending), the original weighted $L_p$ norm models the surface, thereby extending work of \cite{Hyun2017Lp} where the regular rectangular shapes (planar model) had been similarly approximated. Therefore, the surface model for both regular and bent cuboids are implicitly represented as a function of the configuration of the robot and its shape parameters.  

This weighted $L_p$ approach exactly models ellipsoids or circles, as the underlying mathematics are the same. Choosing the proper $p$ and half-lengths of the robot depends on the problem. Here, high $p$ values are chosen to represent cuboid robots and cuboid obstacles. 
 
\subsubsection{Generalized safety constraint of a cuboid in $SE(3)$}
The second contribution of this paper is on the reformulation of the safety constraint by introducing an inner optimization to find the closest point on the cuboid robot to the obstacle in the obstacle's distance metric. The challenge of extending the method in \cite{Hyun2017Lp} to $SE(3)$ is that \emph{edge to edge} collisions in three dimensional space are difficult to model. 
Instead of fixing the candidate points of collision as per \cite{Hyun2017Lp}, we introduce an auxiliary optimization problem which finds the point on the robot with a minimum distance to the obstacle, then checks whether the point belongs to the obstacle or not. Since the surface of the obstacle is also modeled by the $1$-level set of weighted $L_p$ norm, a  minimum distance greater than $1$ implies that the shape and configuration of the robot is safe.  

The challenge is that, now the constraints of the original path planning problem require the solution of another optimization problem. This, so called \emph{two staged optimization}, is avoided by analytically deriving the necessary and sufficient conditions for the closest point problem for both regular and bent cuboid robots. In the end, a total of four equality and two inequality constraints represent the safety constraint of a rigid cuboid robot, and four equality and ten inequality constraints for a bendable cuboid robot. 

\subsubsection{Optimal path planning for rigid and bendable cuboid robot}
Finally, the generalized Bolza type optimal path planning problem is formulated, containing only analytic constraints combined with logical AND operations. Numerical shortest path solutions for rigid regular cuboid robots and bendable cuboid robots are sought by transferring the infinite dimensional problem to a nonlinear programing problem (NLP).  The MATLAB based software OPTRAGEN \cite{Bhattacharya2006} performs the conversion, where the trajectories are defined as B-splines, and constraints are computed at given collocation points. The NLP is then solved with the interior point method using IPOPT 3.12.6 \cite{Wachter2006}.  Collision avoidance is numerically and graphically verified by simulated results. 

{\bf Contents.} The next section (\S \ref{sec:preliminary}) reviews the weighted $L_p$ norm, and the kinematics of rigid body frame in $SE(3)$. The cuboid robot model is introduced in \S\ref{sec:Model}, and the optimal path planning problem is defined in \S\ref{sec:prob}. The analytic approximation of the robot's surface is provided in \S\ref{sec:Surface_Model}, and the safety constraint for collision avoidance is derived in \S\ref{sec:col_avoid}. The shortest safe path is sought numerically in \S\ref{sec:path}.

\section{Preliminary}
\label{sec:preliminary}

\subsection{Weighted $L_p$ norm}
%
%
\label{sec:weighted}

Denote a vector, $x=(x_1,\cdots,x_n)$ in $\mathbb{R}^n$ for $n\in\mathbb{N}$ be a positive vector if $x_i > 0$ for all $i$.

\begin{definition}[Weighted $L_p$ norm]
\label{def:weighted_Lp}
Let $\halfL \in \mathbb{R}^n$ be a positive vector, and $0 < p < \infty$, 
then $||\cdot||_{\halfL, p}:\mathbb{R}^n\to\mathbb{R}_{\geq 0}$ is the
$\halfL$-weighted $L_p$ norm such that
\begin{equation*}
  ||x||_{(\halfL,p)} := \left(\sum_{i=1}^n 
         ({|x_i|}/{\halfL_i})^p\right)^{\frac{1}{p}},
\end{equation*} 
for all $x\in\mathbb{R}^n$ where $\halfL=(\halfL_1,\cdots,\halfL_n)$.  The level set $\{x\in\mathbb{R}^n:||x||_{(\halfL,p)}=1\}$ is called the \emph{unit sphere} of the $L_{\halfL,p}$ norm. 
\end{definition}

As $p$ increases, the $L_{\halfL,p}$ norm unit sphere approaches a rectangular shape in $\mathbb{R}^2$ and a cuboid shape in $\mathbb{R}^3$. More details on level sets for weighted $L_p$ norm can be found in \cite{Hyun2017Lp}.%

\subsection{Kinematics in $SE(3)$}
\label{sec:kinematics}

This paper studies the kinematic model for cuboid robots in $SE(3):=\mathbb{R}^3\times SO(3)$ where $SO(3)$ is the special orthogonal group in $\mathbb{R}^{3\times 3}$ \cite{murray1994mathematical}. Let $\rPose(t)\in SE(3)$ be the robot frame at time $t$,  
\begin{equation} \label{eqn:robotframe}
  \rPose(t):=\rCoordsSEThreeAt{t}^T,
\end{equation} 
where $\roboPos{t}:=\xCoordsAtSEThree{t}$ is the center of mass (CoM) position of the robot in the world frame, and $\rCoordRAt{t}\in SO(3)$ is the rotation matrix describing the orientation of robot frame relative to the world frame. This kinematic model assumes that the robot controls the speed of the CoM, $\roboControlLVel(t)\in\mathbb{R}$, and the instantaneous angular velocity, $\roboControlAVel(t):=\roboControlAVelAt{t}\in\mathbb{R}^3$, in the body frame. These kinematics are 
\begin{eqnarray} 
  \label{eqn:kinematicsSE3}
  \roboPosDer{t}&=&\rCoordRAt{t}\roboVel\roboControlLVel(t) \\
  \label{eqn:kinematicsSE32}
  \dot{\rCoordR}(t)&=&\rCoordRAt{t}\hat{\roboControlAVel}(t)
\end{eqnarray}
where $\hat{\roboControlAVel}(t)$ maps $\roboControlAVel$ to a
$3\!\times\!3$ skew-symmetric matrix in $so(3)$, 
\begin{equation} \label{eqn:skewsymSE3}
  \hat{\roboControlAVel}(t):=\begin{pmatrix}
      0 & -\omega_3 & \omega_2 \\
      \omega_3 & 0 & -\omega_1 \\
       -\omega_2 & \omega_1 & 0
    \end{pmatrix},
\end{equation}
and $\roboVel\in\mathbb{R}^3$ is the linear velocity direction in the
body frame.

%
%

\section{Cuboid Robot Models}
\label{sec:Model}
This section explores the two types of robot geometries considered.  One is a rigid model where the geometry of the robot does not change, and the other is a bendable model where the robot is assumed to be bendable in one of the axis in the robot's frame.  The former is the \emph{regular cuboid} geometry where the rigid shape is fully determined by the half lengths, $\halfL:=(\halfL_1, \halfL_2, \halfL_3)$, on each axis in the robot frame.  Figure~\ref{fig:Lp_levelset_3d} depicts an example with $\halfL=(2,1,1)$.  The latter is the bent cuboid geometry where the bend is along the x-y plane with the constant curvature, $\roboCurvature\in\mathbb{R}$. The full geometry of the \emph{bent cuboid} requires the half lengths, $\halfL:=(\halfL_1, \halfL_2, \halfL_3)$, the curvature, $\roboCurvature$, and the bending angle $\theta_\roboCurvature$.  Figure~\ref{fig:Lp_Bent_levelset_3d_positive} depicts an example with $\halfL=(2,1,1)$, $\roboCurvature=0.3927$, and $\theta_\roboCurvature=\pi/2$.  For both geometries, there exists a curve passing through the robot which serves as an axis of symmetry for its surface.  
\begin{definition} \label{def:centerline}
The \emph{centerline} of the robot is a curve satisfying the following: 
1) An intersection between a normal plane of the curve with the robot at
any point on the curve creates a rectangle with half lengths
$(\halfL_2,\halfL_3)$; and
2) The curve passes through the center of all such rectangles.
\end{definition}

The depictions in Figure~\ref{fig:model} include the robot centerlines, drawn as solid red curves penetrating the robot. The dashed red line on the top surface visualizes the shape of the centerline inside the robot.  The length of the centerline inside the robot is $2\halfL_1$ for the regular cuboid, and $1/\roboCurvature\cdot\theta_\roboCurvature$ for the bent cuboid.  Inspired by the fact that the bent cuboid can be generated by bending the regular cuboid, the next invariant property for centerline is assumed.

\begin{assumption}[Invariant property] \label{asmp:invariant_center}
The length of the centerline (LoC) captured in the bent cuboid is equal to the LoC of the original regular cuboid.
\end{assumption}

The LoC of the examples in Figure~\ref{fig:model} are both $4$. The invariant property implies that the bending angle can be explicitly expressed by the curvature and the half-lengths as 
\begin{equation} \label{eqn:bending_angle}
  \theta_\roboCurvature= 2\halfL_1|\roboCurvature|,
\end{equation}
indicating that the half lengths and the curvature fully characterize
the geometry of the bending robot. Therefore, we denote the \emph{shape
parameters} for the regular cuboid by $\halfL$, and for the bent cuboid
by $(\halfL,\kappa)$.

\begin{figure}[t!]
  \captionsetup[subfigure]{}
  \centering
  \subfloat[Regular cuboid robot]
    {{\includegraphics[width=1.75in, clip=true,trim=0.15in 0in 0.5in 0.5in]{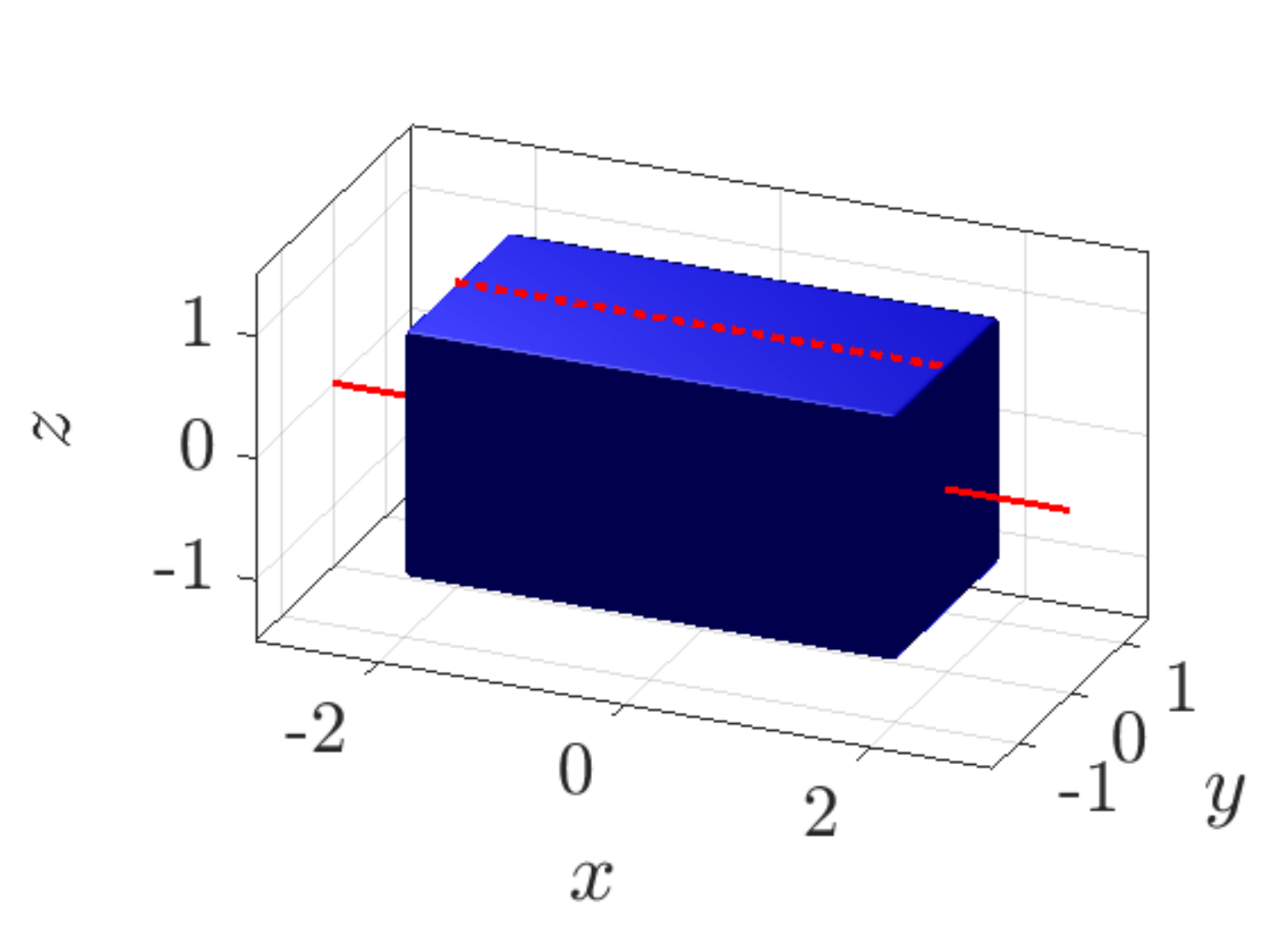}}\label{fig:Lp_levelset_3d}}
  \subfloat[Bent cuboid robot]
		{{\includegraphics[width=1.65in, clip=true,trim=0.55in 2.7in 1.0in 2.7in]{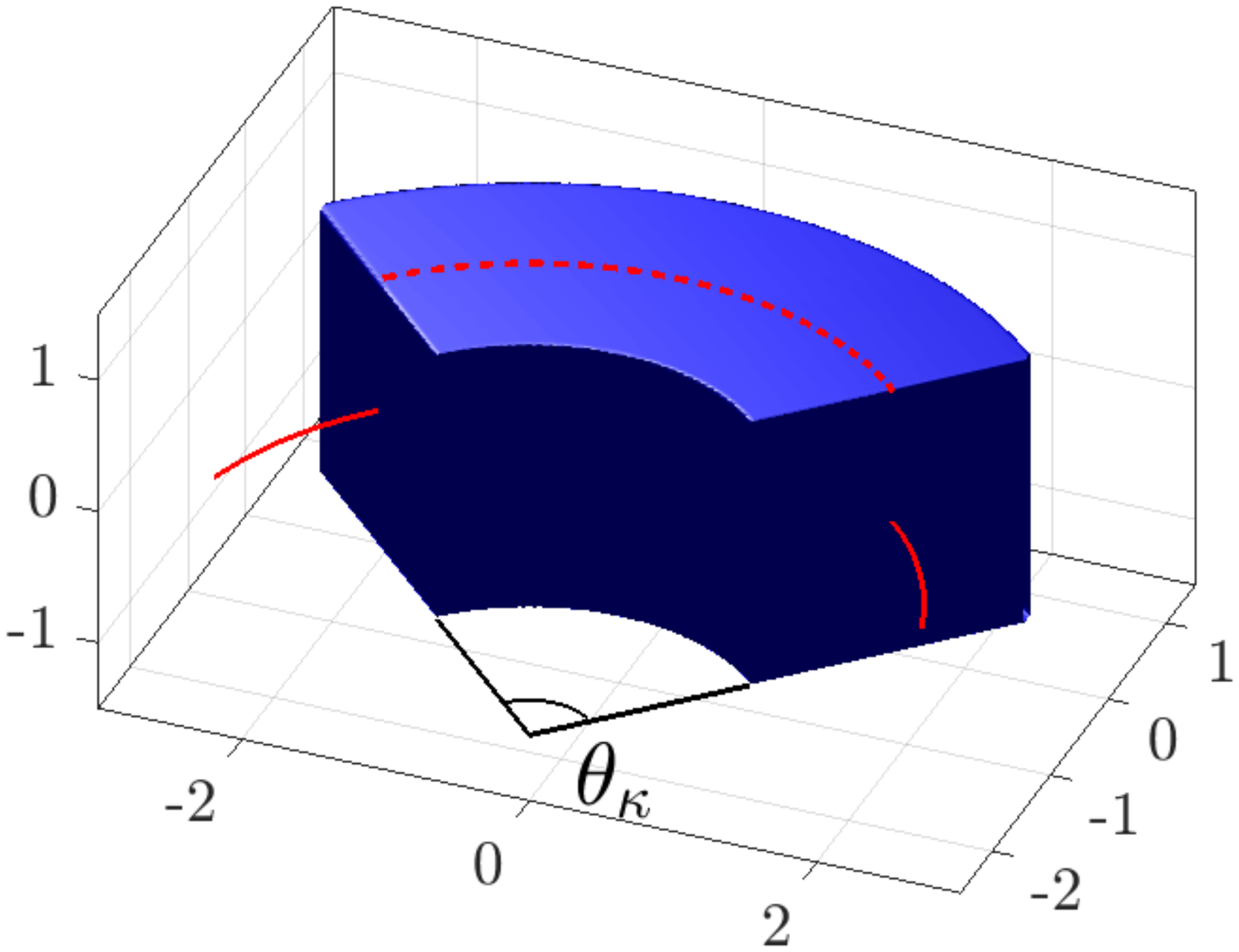}}\label{fig:Lp_Bent_levelset_3d_positive}}
	\caption{Two types of geometries for cuboid robots in $SE(3)$}
\label{fig:model}
\end{figure}

\subsection{Rigid model}
The rigid model shape is time invariant. Therefore, the shape parameters
$\halfL$ and $\roboCurvature$ are fixed, and the required controls for
rigid body kinematics are $\roboControlLVel$ and $\roboControlAVel$. The
regular cuboid model could represent a box shaped furniture in the 3D
space, and the bent cuboid model could be used as an surface model for a
boomerang. 

\subsection{Bendable model}
The bendable model admits time variable shape parameters. Therefore, the
shape parameters $\halfL(t)$ and $\roboCurvature(t)$ become additional
control parameters besides $\roboControlLVel$ and $\roboControlAVel$. As
we consider length preserving robots, the following assumptions holds:
\begin{assumption}[Constant length assumption] \label{asmp:const_length}
The half lengths of the robot $\halfL>0$ are fixed, and the robot can be
bent by controlling the curvature $\roboCurvature(t)$.
\end{assumption}

Though the robot may bend with different curvatures, the above assumption means that the LoC is constant and equal to $2\halfL_1$ for all time, due to (\ref{eqn:bending_angle}).

The bendable cuboid robot is directly related to the fundamental segment of finger motion using fiber-reinforced soft actuators in \cite{Connolly03012017}, and it is also related to the bending motion of the black knifefish in \cite{postlethwaite2009optimal}.   

\textbf{Remark 1}
The volume of the regular cuboid robot, $V_r>0$, with $\halfL$, is
\begin{equation} \label{eqn:constVolume}
  V_r=8\halfL_1\halfL_2\halfL_3.
\end{equation}
From (\ref{eqn:constVolume}), the volume of a bent cuboid with the same half-lengths and satisfying Assumption~\ref{asmp:const_length} is
\begin{eqnarray*}
  V_b(t) 
    &=& ((1/|\roboCurvature(t)|+\halfL_2)^2-(1/|\roboCurvature(t)|
                               -\halfL_2)^2)\theta_\roboCurvature(t)\halfL_3 \\
    &=& 4(1/|\roboCurvature(t)|)\halfL_2\theta_\roboCurvature(t)\halfL_3 \\
    &=& V_r,
\end{eqnarray*}
where the last equality holds by using (\ref{eqn:bending_angle}).  Therefore, the volume is invariant during bending, and the length preserving robot satisfying Assumption~\ref{asmp:const_length} also preserves the volume of the robot. This could be important for liquid filled robots. The proposed algorithm in this paper works equally will in the case of a time varying half-length (i.e., a volume changing robot) without modifying any key result in this paper. 


%

%
%

\section{Problem Statement}
\label{sec:prob}

Consider the optimal path planning problem for cuboid robots navigating 
through a region containing regular cuboid obstacles.
The robot state includes its frame $\rPose(t)\in SE(3)$ and its shape 
parameters $(\halfL(t),\roboCurvature(t))$, which satisfy 
Assumption~\ref{asmp:invariant_center}.
There are $\Nrect$ regular cuboid obstacles each of whose frames and half-lengths are 
$\sPose^j := \sCoordsSupSEThree{j}\in SE(3)$ 
and 
$\halfLSq^j:=(\halfLSq_1^j,\halfLSq_2^j, \halfLSq_3^j)$,
respectively, for 
$j \in \{1, \cdots, \Nrect\}$.
The origin of each obstacle frame is at the center of cuboid. 
An initial $SE(3)$ configuration and a final $SE(3)$ configuration are
given as the path boundary constraints.  The arrival time, $T_f>0$, is
free.

Let $\applyG{\rPose}{\pVar} \in \mathbb{R}^3$ denote the transformation
of coordinates of $\pVar \in \mathbb{R}^3$ according to $\rPose \in
SE(3)$. The inverse $g^{-1}$ would perform the change of coordinates in
the inverse direction.  
Per \S\ref{sec:weighted}, the $\halfLSq^j$-weighted
$L_\infty$ norm represents the regular cuboid with half lengths,
$\halfLSq^j$, and the collision-free spaces of the regular cuboid obstacles
are given by the super-level sets 
\begin{equation*}
R_j  := \{ \pVar \in \mathbb{R}^3 : \left| \left| [\sPose^j]^{-1} (\pVar)
  \right| \right|_{\halfLSq^j,\infty} > 1 \},
\end{equation*}
in the world frame for all $i$ and $j$. 
As the robot moves, the collision constraints in the robot's body frame will change.  The collision-free space in robot body coordinates is
\begin{equation*}
  R^{\roboF}_j  := \{ \pVar \in \mathbb{R}^3 : \left| \left|
                     [\sPose^j]^{-1}\circ \applyG{\rPose}{\pVar}
                                   \right| \right|_{\halfLSq^j,\infty} > 1 \}.
\end{equation*}

Let $B(\halfL(t),\roboCurvature(t))\subset\mathbb{R}^3$ describe the robot's full body in the robot frame at time $t$. The optimal safe path planning problem formulation is
\begin{align}
  \label{eqn:problem_statement_cost}
  \arg \min_{\roboControlLVel, \roboControlAVel, \halfL(t),\roboCurvature(t)} & \int_0^{T_f} L(\rPose, \dot{\rPose}, u, \roboCurvature)dt \ \ \text{subject to} \\
  \label{eqn:problem_statement_equality}
  &\begin{cases}
    \dot{\rPose}            = f(\rPose, \roboControlLVel,\roboControlAVel),\\
    \rPose(0)               = \mathbf{\rCoordsSEThreeSub{i}},\\
    \rPose(T_f)   					= \mathbf{\rCoordsSEThreeSub{f}},\\
  \end{cases}\\
  \label{eqn:problem_statement_inequality}
  & B(\halfL(t),\roboCurvature(t)) \subset(\bigcap_{j=1}^{\Nrect}R_j^\roboF),
\end{align} 
where the vector field $f$ is given by the kinematic equations (\ref{eqn:kinematicsSE3}-\ref{eqn:kinematicsSE32}), $\roboControlLVel\in\mathbb{R}$ and $\roboControlAVel\in\mathbb{R}^3$ are the control inputs, and $L(\rPose, \dot{\rPose}, u)$ is some physically meaningful cost such as control energy or path length.  Depending on the flexibility of the robot, $B(\halfL(t),\roboCurvature(t))$ would be constant (rigid model) or be time varying (bendable model). 

The optimization problem will be solved for the two cases:
\begin{enumerate}
\item (Rigid robot) Find the kinematic controls for fixed shape parameters. 
\item (Bendable robot) Find the kinematic controls and the curvature control, $\roboCurvature(t)$, subject to Assumption~\ref{asmp:const_length}.
\end{enumerate}

Since the constraints in (\ref{eqn:problem_statement_inequality}) typically require OR operations using the hyperplanes of the faces, our goal is to first find a set of inequality and equality constraints with AND operations for the safety constraints.  
\\
{\bf Remark 2.} The obstacles are not restricted to cuboid shapes, rather any $1$-level set of different $\halfLSq^j$ and $p$ value serves to model an obstacle. For example, an ellipsoidal obstacle has $p=2$ and $\halfLSq^j= (a,b,c)$ for distinct $a,b,c>0$, while a spherical obstacle further constrains the $a,b,c$ to be equal. Since safety conditions are derived based on the abstract shape parameters, $\halfLSq^j$, and $p$, the results naturally applies to ellipsoidal and spherical obstacles.

%
%
\begin{figure*}[t!]
  \captionsetup[subfigure]{}
  \centering
  \hfill
  \subfloat[$1$-levelset with $p=2$]
    {{\includegraphics[width=1.70in, clip=true,trim=0.15in 0in 0.55in 0.45in]{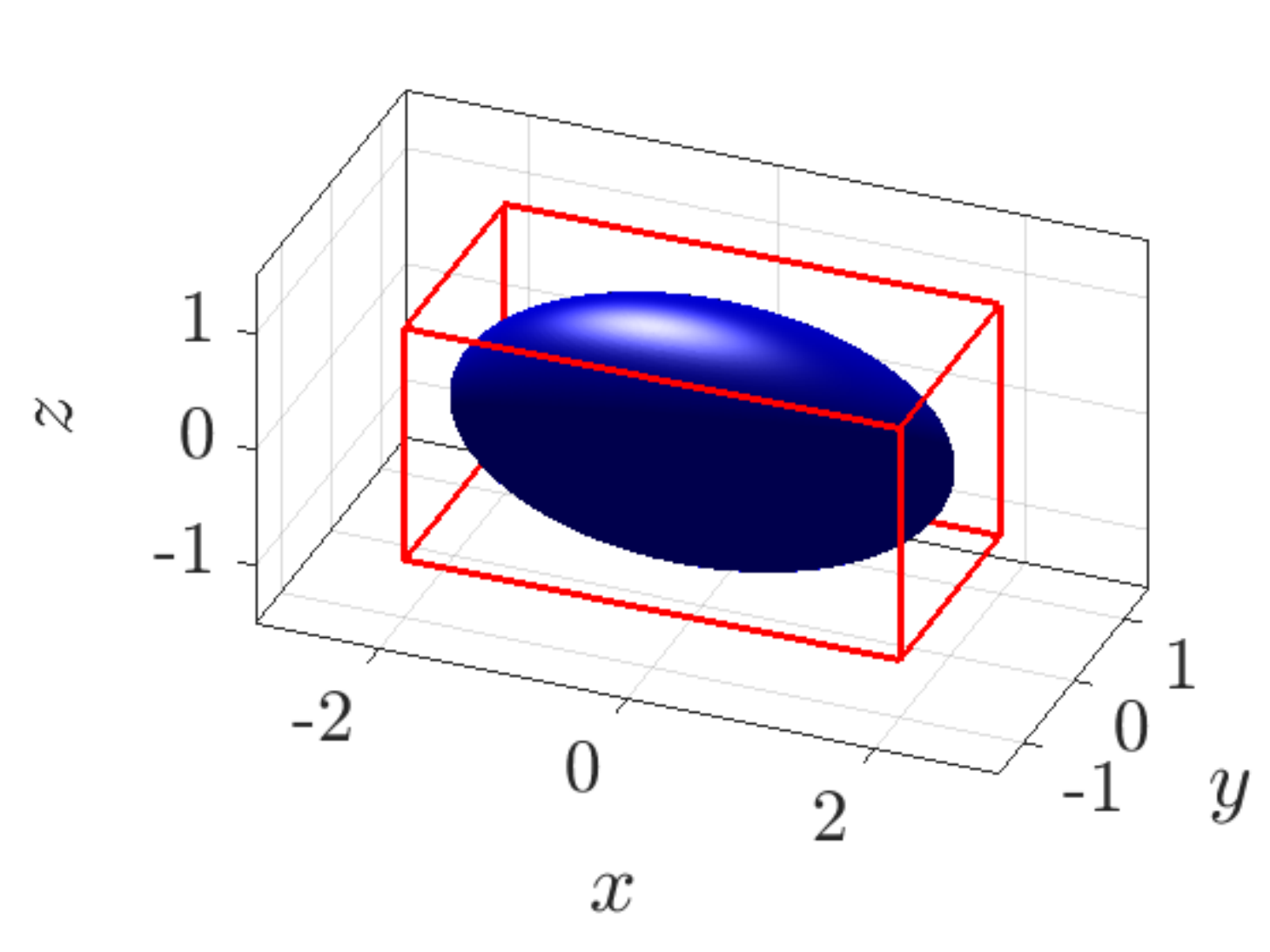}}\label{fig:Lp_levelset_3d_p2}}
  \subfloat[$1$-levelset with $p=10$]
    {{\includegraphics[width=1.7in, clip=true,trim=0.15in 0in 0.55in 0.45in]{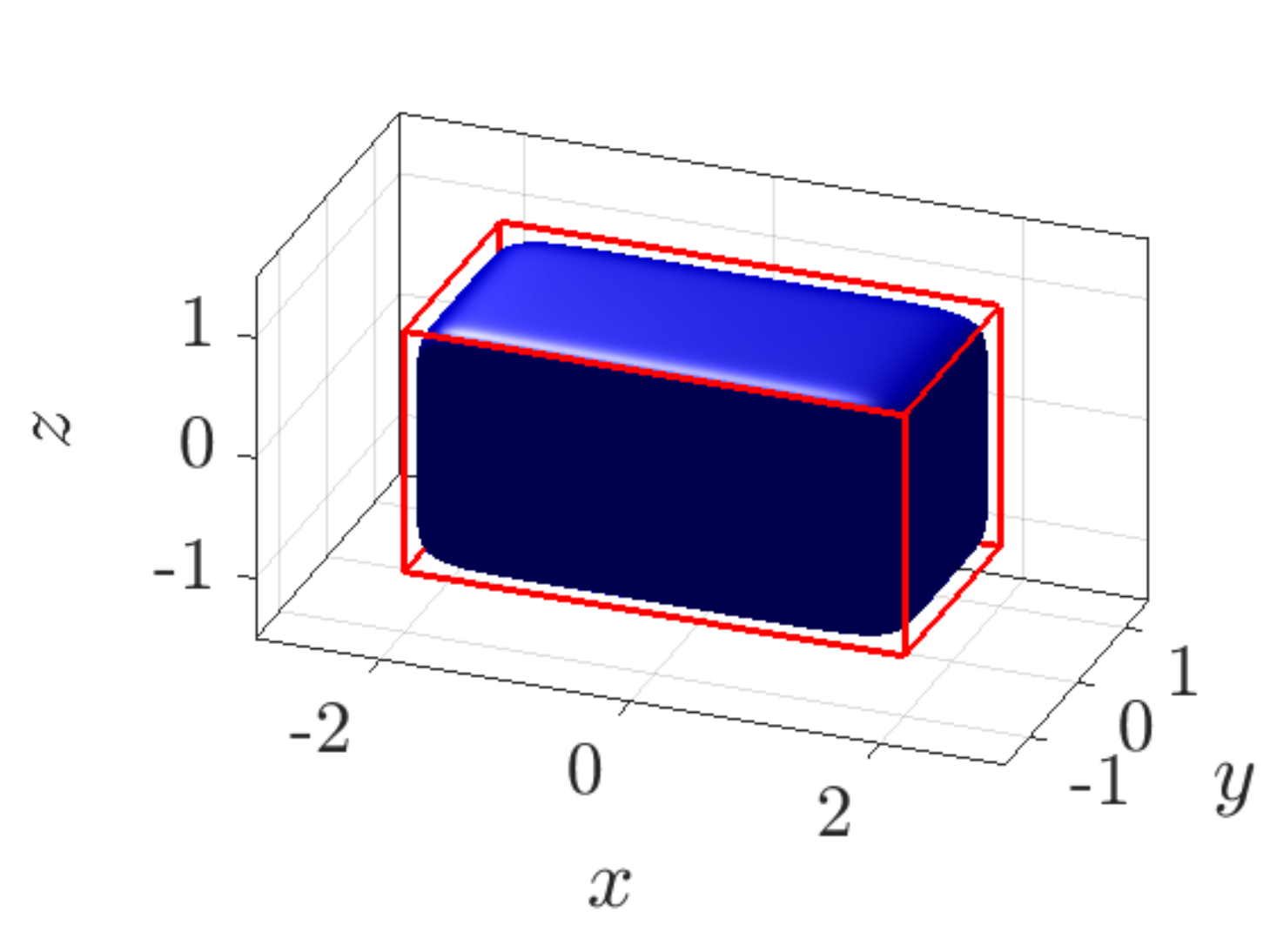}}\label{fig:Lp_levelset_3d_p10}}
		\subfloat[$|\roboCurvature|$-levelset with $p=4$]
    {\includegraphics[width=1.70in, clip=true,trim=0.15in 0in 0.55in 0.45in]{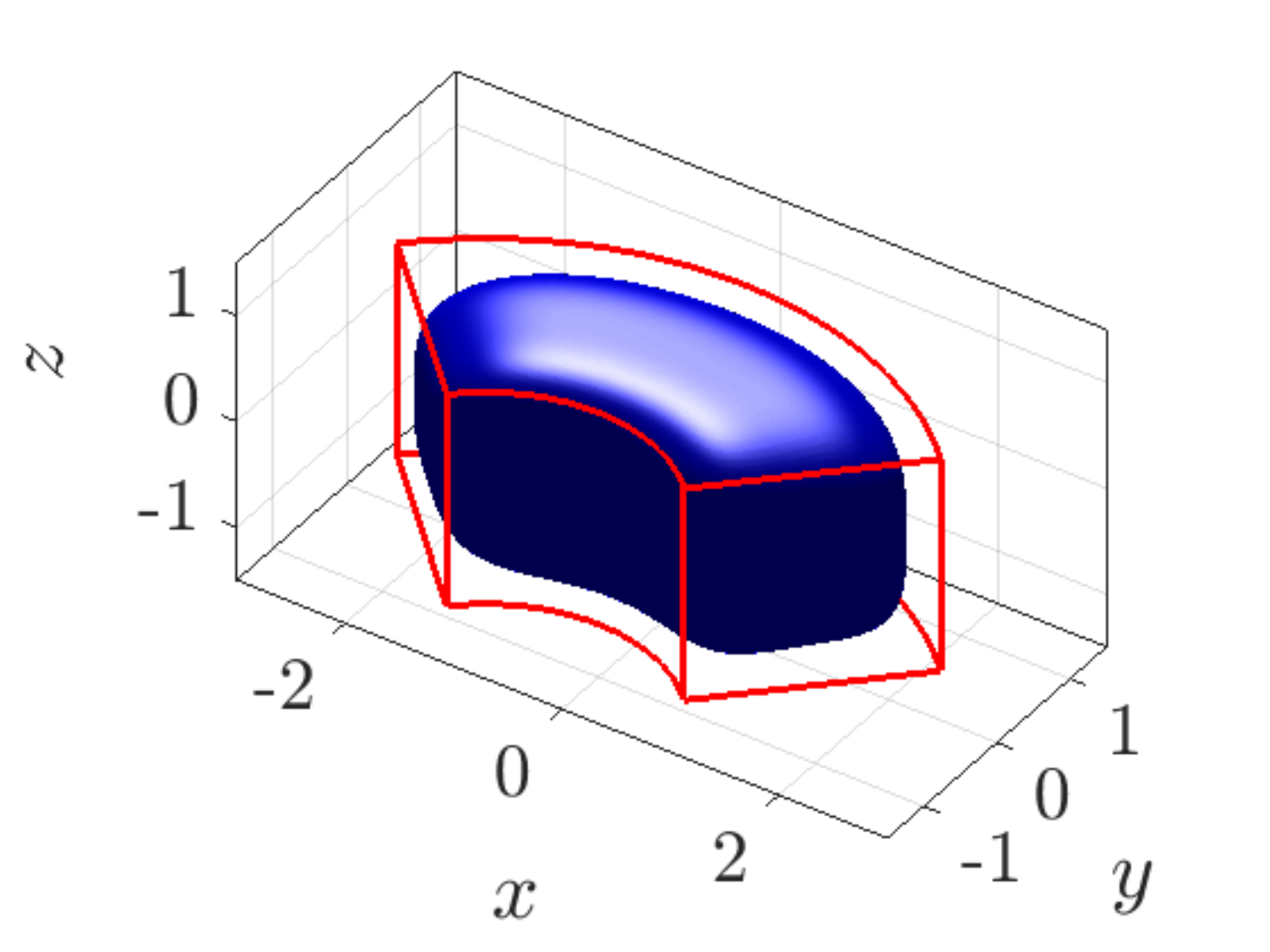}\label{fig:Lp_Bent_levelset_3d_positive_P4}}
  \subfloat[$|\roboCurvature|$-levelset with $p=10$]
    {\includegraphics[width=1.70in, clip=true,trim=0.15in 0in 0.55in 0.45in]{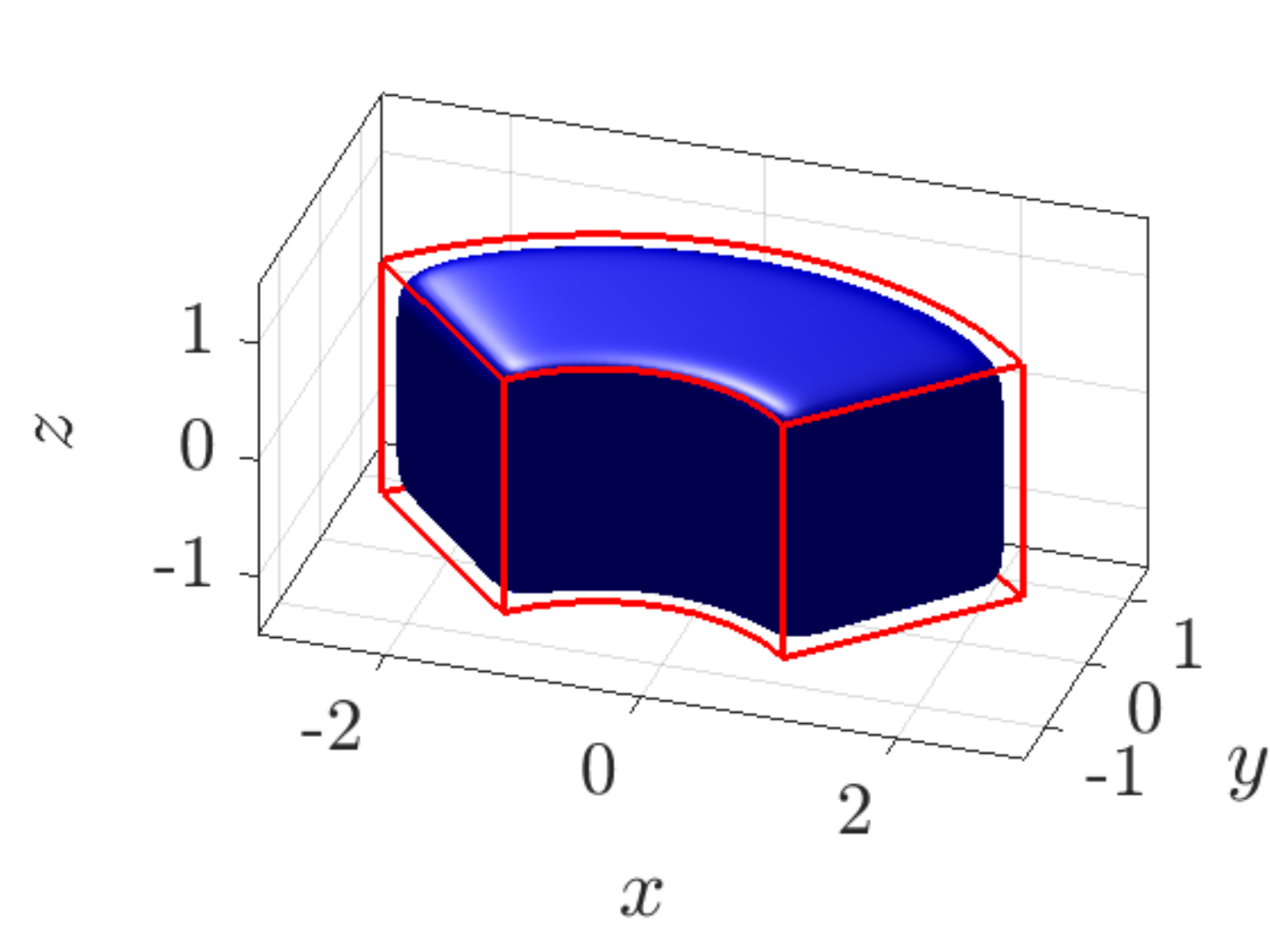}\label{fig:Lp_Bent_levelset_3d_positive_P10}}
\caption{Approximated surface for regular and bent cuboid}
\label{fig:LpLevelset3D}
\end{figure*}

\section{Approximation of Cuboid Robots}
\label{sec:Surface_Model}

In this section, the regular cuboid is approximated by the weighted $L_p$ norm in Cartesian coordinate, and a new positive definite function is proposed, whose $|\roboCurvature|$-level set approximates the bent cuboid. 

\subsection{Regular cuboid approximation}

The rectangular robot surface in $\mathbb{R}^2$ approximated
by the $1$-level set of a weighted $L_p$ norm in \cite{Hyun2017Lp}, 
generalizes to $\mathbb{R}^3$.
Choosing the positive half lengths $\halfL:=(\halfL_1, \halfL_2, \halfL_3)$,
the cuboid surface description is
\begin{equation} \label{eqn:regular_cuboid}
  B_{(\halfL,p)}:=\{\vPose\in\mathbb{R}^3| \text{ } ||\vPose||_{(\halfL,p)}=1\}.
\end{equation}
Equation (\ref{eqn:regular_cuboid}), evaluated in the robot frame, will play a role in formalizing the safety constraint (\ref{eqn:problem_statement_inequality}) using inequality constraints.
Figure~\ref{fig:LpLevelset3D} consists of several examples of cuboid approximations using the weighted $L_p$ approximation with different $p$ values. 
The half lengths are chosen as $\halfL=(2,1,1)$, with Figure~\ref{fig:Lp_levelset_3d_p2} visualizing an ellipsoid for the choice $p=2$, and Figure~\ref{fig:Lp_levelset_3d_p10} an approximate cuboid for the choice $p=10$.
The surface model approaches that of a regular cuboid (the red boundaries in
Figure~\ref{fig:LpLevelset3D}) as the value of $p$ increases.  The example
of Figure~\ref{fig:Lp_levelset_3d} uses $p=200$.

%

\subsection{Bent cuboid approximation}
\label{sec:Surface:Bend_cuboid}
\begin{figure}[t!]
  \centering
  \hfill
	\subfloat[(top) $\roboCurvature>0$ and (bottom) $\roboCurvature<0$]{
	\begin{tikzpicture}
\node[anchor=south] (russell.south)
    {\includegraphics[width=.22\textwidth, clip=true,trim=0.1in 0.3in 0.9in 1.0in]{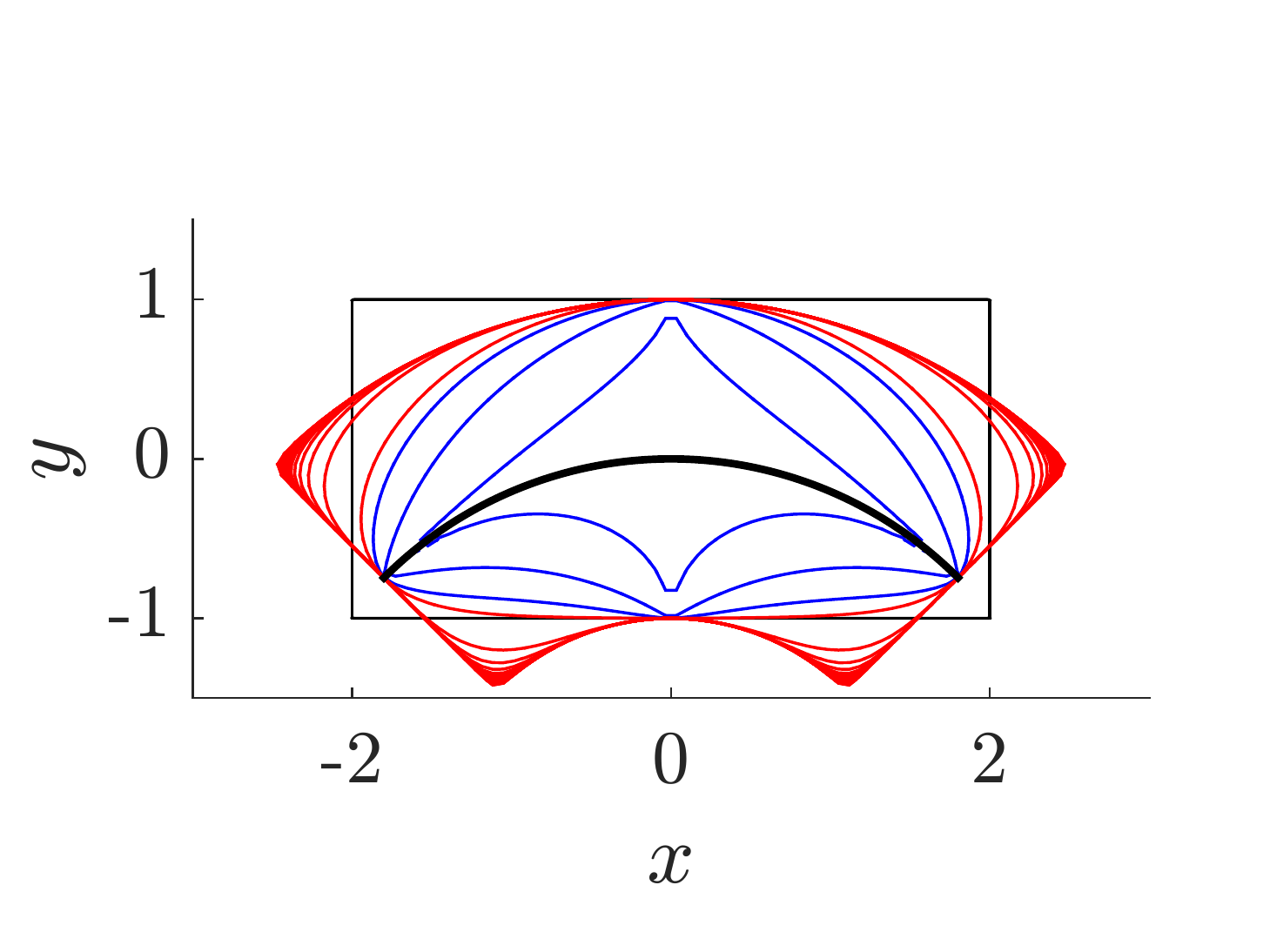}};
\node[anchor=north] (whitehead.south)
    {\includegraphics[width=.22\textwidth, clip=true,trim=0.1in 0.3in 0.9in 1.0in]{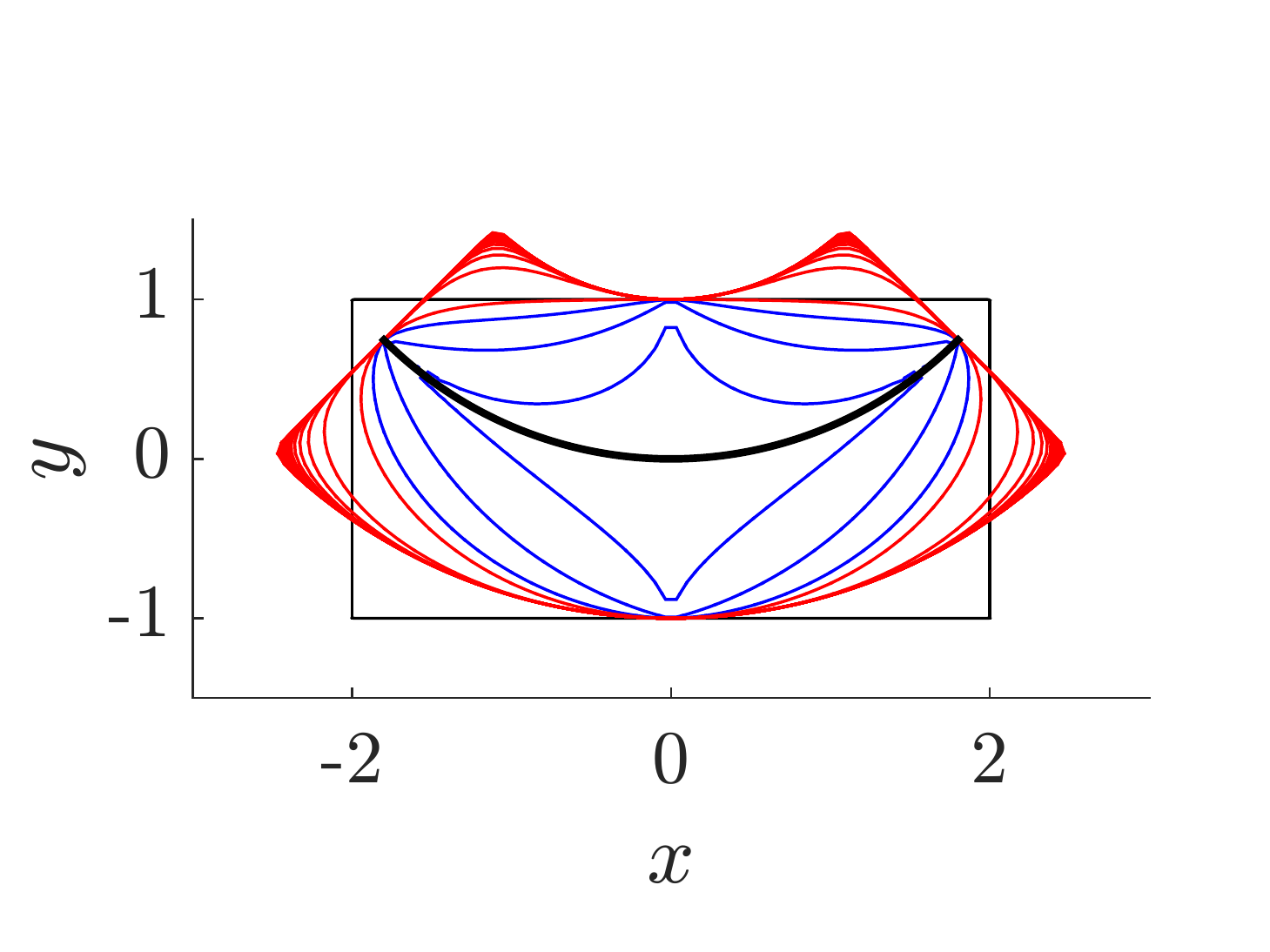}};
\end{tikzpicture}\label{fig:Lp_Bent_levelset_positive_negative}}
\subfloat[Level sets for $p=10$]
    {{\includegraphics[width=0.25\textwidth, clip=true,trim=0.5in 0.1in 0.4in 0.3in]{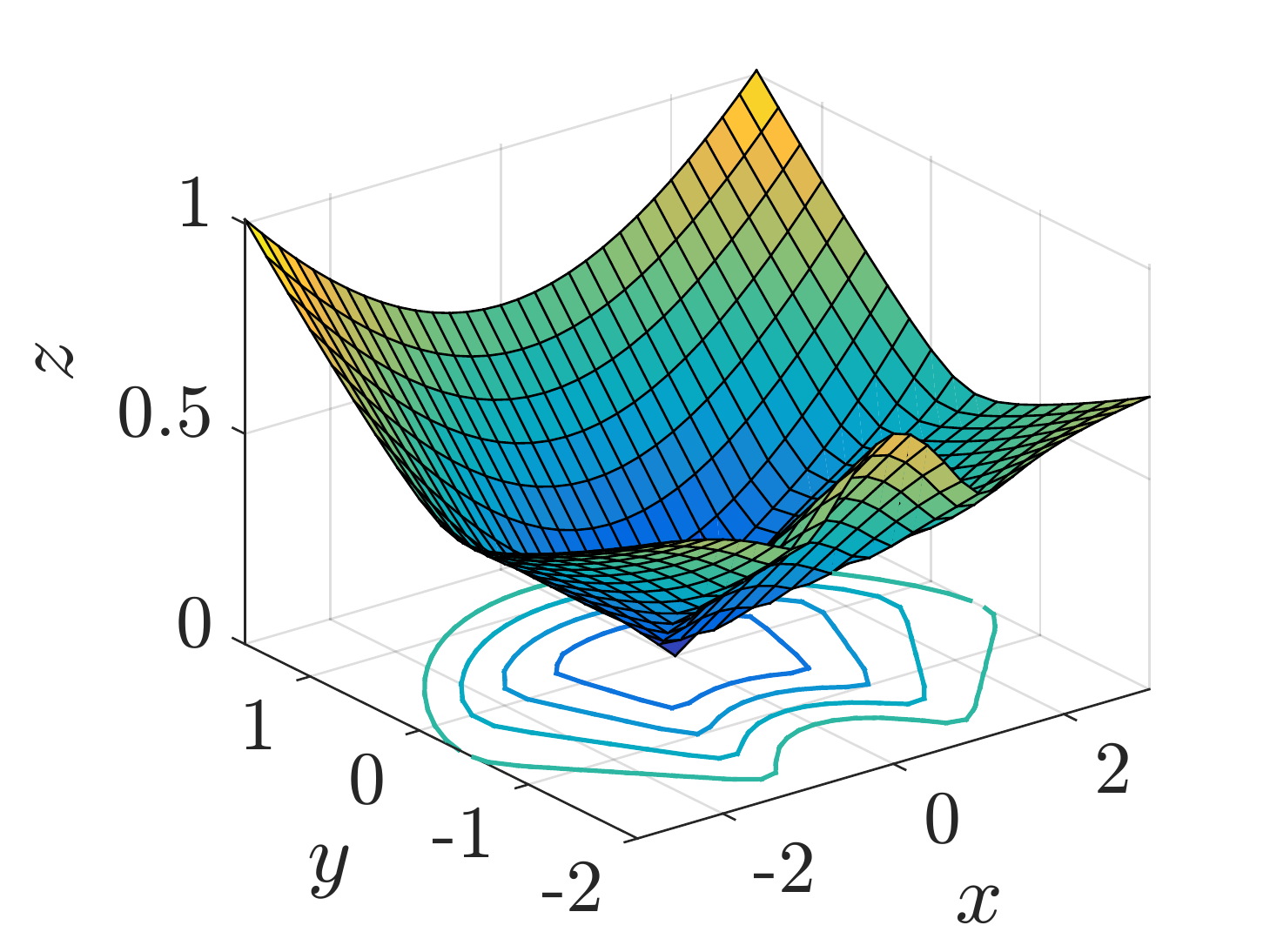}} \label{fig:Lp_Bent_levelset_norm}}
		
	\caption{Geometrical interpretation of weighted polar $L_p$ in $2D$}
\label{fig:polar_Lp_2d}
\end{figure}

To approximate bent cuboids, this section introduces a weighted $L_p$ norm
in polar coordinates.  Analysis of the planar bent rectangle case
establishes the important geometry associated to the polar space and its
connection to the original shape.  Extending the planar results to
$\mathbb{R}^3$ will then involve extruding the planar bent rectangle in the
vertical direction, normal to the plane.  A modified weighted, polar $L_p$
norm recovers the $3D$ surface for bent cuboids.

\subsubsection{Bent rectangle}
Let $\roboCurvature\neq 0$ be given, and $\halfL$ be fixed. 
Define the coordinate transformation 
$\mathcal{T}_\roboCurvature:\mathbb{R}^2\to\mathbb{R}^2$ to be
\begin{equation} \label{eqn:curvature_isomorphism}
  \mathcal{T}_\roboCurvature(\begin{pmatrix}\vCoordX \\ \vCoordY\end{pmatrix})
    := \begin{pmatrix}\roboCurvature\vCoordX \\ \roboCurvature\vCoordY+1
       \end{pmatrix},
\end{equation}
for every $\vPose:=(\vCoordX,\vCoordY)\in\mathbb{R}^2$.  The polar coordinates of $\vPose$, transformed by $\mathcal{T}_\roboCurvature$, are
\begin{eqnarray} 
  \label{eqn:R_polar}
  R_{\mathcal{T}_\roboCurvature}(\vPose) &:=&
    \sqrt{(\roboCurvature\vCoordX)^2+(\roboCurvature\vCoordY+1)^2},
  \\
  \label{eqn:theta_polar}
  \theta_{\mathcal{T}_\roboCurvature}(\vPose) & := &
    \arctan{(\left(\roboCurvature\vCoordY+1)/\roboCurvature\vCoordX\right)}.
\end{eqnarray} 
\begin{definition}[Weighted polar $L_p$] \label{defn:polar_Lp_}
  The \emph{$(\halfL, \roboCurvature)$-weighted polar $L_p$} function is the
  positive definite function 
  $\Phi_{(\halfL, \roboCurvature, p)}:\mathbb{R}^2\to\mathbb{R}$,
  \begin{equation} \label{eqn:polar_Lp_SE2}
  \Phi_{(\halfL, \roboCurvature, p)}(\vPose) :=
    ( ({|R_{\mathcal{T}_\roboCurvature}(\vPose)-1|}/{\halfL_2})^p
    + (|\theta_{\mathcal{T}_\roboCurvature}(\vPose)-\theta_0|/\halfL_1)^p
    )^{1/p},
\end{equation}
where $\theta_0=\sign(\roboCurvature)\cdot\pi/2$.
\end{definition}
The function $\sign(\roboCurvature)=1$ if $\roboCurvature\geq 0$ and $-1$ if $\roboCurvature<0$.  When the parameters $(\halfL, \roboCurvature)$ are known, we will call the value function of Definition~\ref{defn:polar_Lp_} the weighted polar $L_p$ function. 

The surface of the bent rectangular robot is approximated by the $|\roboCurvature|$-level set of $\Phi_{(\halfL,  \roboCurvature, p)}$. Several examples of the $|\roboCurvature|$-level set of $\Phi_{(\halfL, \roboCurvature, p)}$ with different $p$ values are shown in Figure~\ref{fig:Lp_Bent_levelset_positive_negative}, where $\halfL=(2,1)$ is chosen with $\roboCurvature=0.3927$ and $\roboCurvature=-0.3927$, respectively. The blue contours represents the $p$ value changes from $0.6$ (innermost) to $1.6$ (outermost) with step size $0.5$, and the red contours represents the changes from $2$ (innermost) to  $40$ (outermost) with step size $2$. 

The $((2,1), 0.3927)$-weighted polar $L_p$ function with $p=10$ and its different level sets are shown in Figure~\ref{fig:Lp_Bent_levelset_norm} where the contours projected on $(x,y)$ plane corresponds to the levels from $0.2$ (innermost) to $0.5$ (outermost). 

\begin{lemma} \label{lem:polar_Lp}
  For $p \ge 2$ even, 
  $\Phi_{(\halfL, \roboCurvature, p)}$ is a positive definite function;
\end{lemma}
\begin{proof}
Since $p$ is even number, $\Phi_{(\halfL, \roboCurvature,
p)}(\vPose)\geq 0$ holds for every $\vPose\in\mathbb{R}^2$. Also
$\Phi_{(\halfL, \roboCurvature, p)}(\textbf{0})=0$ holds for zero vector,
$\textbf{0}\in\mathbb{R}^2$ from
$\lim_{\alpha\to\infty}\arctan{\alpha}=\pi/2$. 
Now, suppose that $\Phi_{(\halfL, \roboCurvature, p)}(\vPose)= 0$ holds,
then $\theta_{\mathcal{T}_\roboCurvature}(\vPose)=\theta_0$ and
$R_{\mathcal{T}_\roboCurvature}(\vPose)=1$ follows from $p$ even.
From the first observation $\arctan{(\left(\roboCurvature\vCoordY+1)/\roboCurvature\vCoordX\right)}=\sign(\roboCurvature)\pi/2$,
which implies $\vCoordX=0$. From the second observation, $\vCoordY=0$. 
Therefore, $\Phi_{(\halfL, \roboCurvature, p)}(\vPose)= 0$ holds
if and only if $\vPose=\textbf{0}$. Hence, $\Phi_{(\halfL, \roboCurvature,
p)}$ is a positive definite function.
\end{proof}
%
%
With regards to collision detection, 
evaluation of the weighted polar $L_p$ function at a point inside the robot is 
less than $|\roboCurvature|$, and at a point outside is greater than 
$|\roboCurvature|$. 
Additionally, the $|\roboCurvature|$-level curves do not self intersect. 
Therefore, all safe configurations of the bent rectangular robot to a point
obstacle are simply obtained by evaluating the weighted $L_p$ function at
the obstacle point in the robot frame and checking if it is greater than
$|\roboCurvature|$. 
The $|\roboCurvature|$-level set of $(\halfL,\roboCurvature)$-weighted polar
$L_p$ function approaches the bent rectangular robot for $p \rightarrow
\infty$. 
\begin{theorem} \label{thm:polar_Lp_2D}
The boundary of the bent rectangular robot, $\partial B_{((\halfL_1,\halfL_2),\roboCurvature)}$, is equivalent to the $|\roboCurvature|$-level set of $\Phi_{(\halfL,\roboCurvature,p)}$, 
\begin{equation}
\label{eqn:kappa_level_2D}
\partial B_{((\halfL_1,\halfL_2),\roboCurvature, p)}:=\{\vPose\in\mathbb{R}^2| \Phi_{(\halfL,\roboCurvature,p)}(\vPose)=|\roboCurvature| \},
\end{equation}
as $p$ approaches $\infty$
\end{theorem}
\begin{proof}
See APPENDIX \ref{app:polar}.
\end{proof}
The following proposition shows that Assumption~\ref{asmp:invariant_center} holds for the approximate, weighted polar $L_p$ model. 
\begin{proposition} \label{prop:bend_rectangular}
  The LoC of $\partial B_{((\halfL_1,\halfL_2),\roboCurvature,p)}$ is equal to the LoC of original rectangular robot with half-lengths $(\halfL_1,\halfL_2)$. 
\end{proposition}
\begin{proof}
APPENDIX~\ref{app:polar} defines the centerline for the $((\halfL_1,\halfL_2),\roboCurvature)$-weighted polar $L_p$ function as a parameterized curve, (\ref{eqn:bent_centerline}).  The centerline is an arc with radius of curvature $1/|\roboCurvature|$, whose polar angle varies in the range $[\sign(\roboCurvature)\pi/2-\theta_\roboCurvature/2,\sign(\roboCurvature)\pi/2+\theta_\roboCurvature/2]$ with $\theta_\roboCurvature$ the desired angle of curvature in (\ref{eqn:bending_angle}). Computing the length of the arc, the LoC of $\partial B_{((\halfL_1,\halfL_2),\roboCurvature,p)}$ is found to be $2\halfL_1$, as needed by Assumption~\ref{asmp:invariant_center}.  \end{proof}

Centerlines examples are shown as black curves in
Figure~\ref{fig:Lp_Bent_levelset_positive_negative} for positive and negative
curvatures, respectively.

\subsubsection{Bent cuboid}

In this section, the bent cuboid surface model is approximated by vertically expanding the bent rectangle in the robot's frame. Inspired by the fact that the different level sets still preserve the shape of the bending effect as shown in Figure~\ref{fig:Lp_Bent_levelset_norm}, the multiple level of $\Phi_{(\halfL,\roboCurvature,p)}$ will be considered as a $X-Y$ section at different $z$ values. 

Without loss of generality, let the bent rectangular robot with the shape
parameters, $((\halfL_1,\halfL_2),\roboCurvature)$, be the $X-Y$ section at
$z=0$ of the bent cuboid robot with the shape parameters,
$((\halfL_1,\halfL_2,\halfL_3),\roboCurvature)$ in the robot's frame. 

A new positive definite function is given as follows. 
\begin{definition}[Weighted polar $L_p$ in $3D$]
\label{defn:polar_Lp_SE3}
The \emph{$(\halfL, \roboCurvature)$-weighted polar $L_p$ function in $3D$}.  is the positive definite function, 
$\Psi_{(\halfL, \roboCurvature, p)}:\mathbb{R}^3\to\mathbb{R}$, 
\begin{equation} \label{eqn:polar_Lp_SE3}
  \Psi_{(\halfL, \roboCurvature, p)}(\vPose) :=
   ( (\Phi_{((\halfL_1,\halfL_2), \roboCurvature, p)}(\vCoordX,\vCoordY))^p
     + (|\roboCurvature|\vCoordZ/\halfL_3)^p)^{1/p},
\end{equation}
where $\vPose=(\vCoordX,\vCoordY,\vCoordZ)\in\mathbb{R}^3$.
\end{definition}
Similar to the bent rectangle case, the $|\roboCurvature|$-level set of $\Psi_{(\halfL, \roboCurvature, p)}$ represents the approximated bent cuboid. Several examples of the $|\roboCurvature|$- level sets of $\Psi_{((2,1,1),\roboCurvature, p)}$ for $\roboCurvature=0.3927$ with different $p$ values are shown in Figure~\ref{fig:Lp_Bent_levelset_3d_positive_P4}-\ref{fig:Lp_Bent_levelset_3d_positive_P10}. The red boundary represents the original bent cuboid robot, and the blue surface represents the approximation with different $p$ values. As $p$ increases higher, the approximated bent cuboid model approaches closer to the original bounds. An example when $p=200$ is shown in Figure~\ref{fig:Lp_Bent_levelset_3d_positive}.
The $|\roboCurvature|$-level set of $\Psi_{(\halfL, \roboCurvature, p)}$ is a proper approximation to the bent cuboid robot in the robot's frame.

\begin{theorem} \label{thm:polar_Lp_3D}
  The surface of the bent cuboid robot, $\partial B_{((\halfL_1,\halfL_2,\halfL_3),\roboCurvature)}$, is equivalent to the $|\roboCurvature|$-level set of $\Psi_{(\halfL,\roboCurvature,p)}$, 
  \begin{equation} \label{eqn:kappa_level_3D}
    \partial B_{((\halfL_1,\halfL_2,\halfL_3),\roboCurvature, p)}
      := \{\vPose\in\mathbb{R}^3 | 
          \Psi_{(\halfL,\roboCurvature,p)}(\vPose)=|\roboCurvature| \},
  \end{equation}
  as $p$ approaches $\infty$
\end{theorem}
\begin{proof}
It is enough to show that, for every $\vCoordZ\in [-\halfL_3,\halfL_3]$, the
approximated surface $\partial B_{((\halfL_1,\halfL_2),\roboCurvature,
p)}^{\vCoordZ}$ defined by
\begin{equation} \label{eqn:theorem2}
  \{\vPose\in\mathbb{R}^3|\Phi_{((\halfL_1,\halfL_2), \roboCurvature, p)}
    (\vCoordX,\vCoordY)=|\roboCurvature|(1-(\vCoordZ/\halfL_3)^p)^{1/p}\}
\end{equation}
is a bent rectangle with shape parameters $((\halfL_1,\halfL_2),\roboCurvature)$, as $p \rightarrow \infty$. If $|\vCoordZ|/\halfL_3< 1$, then
\begin{equation*}
  \lim_{p\to\infty} |\roboCurvature|(1-(\vCoordZ/\halfL_3)^p)^{1/p}=|\roboCurvature|,
\end{equation*}
 and so $\partial B_{((\halfL_1,\halfL_2),\roboCurvature, p)}^{\vCoordZ}$ is the $|\roboCurvature|$-level set of $\Phi_{((\halfL_1,\halfL_2), \roboCurvature, p)}$. Therefore, by Theorem~\ref{thm:polar_Lp_2D}, the level set is the bent rectangle with the desired shape parameters. 

If $|\vCoordZ|/\halfL_3=1$, then 
\begin{equation*}
\lim_{p\to\infty} \Psi_{(\halfL, \roboCurvature, p)}(\vPose)=\max{(\Phi_{((\halfL_1,\halfL_2), \roboCurvature, p)}(\vCoordX,\vCoordY), |\roboCurvature\vCoordZ|/\halfL_3)},
\end{equation*} 
by using the fact the $L_\infty$ norm is equivalent to the maximum absolute
coordinate. 
Therefore, the approximated surface, $\partial B_{((\halfL_1,\halfL_2),\roboCurvature, p)}^{\vCoordZ}$, is equal to the sublevel set, 
\begin{equation*}
\{\vPose\in\mathbb{R}^3|\Phi_{((\halfL_1,\halfL_2), \roboCurvature,
\infty)}(\vCoordX,\vCoordY)\leq|\roboCurvature| \text{ } \;\wedge\;  |\vCoordZ|/\halfL_3=1\},
\end{equation*}
which consists of the top and bottom faces of the bent cuboid as 
$p \rightarrow \infty$. 
\end{proof}
It follows that the approximated bent cuboid satisfies 
Assumption~\ref{asmp:invariant_center}. 
\begin{corollary} \label{coro:bend_cuboid}
  The LoC of $\partial B_{((\halfL_1,\halfL_2,\halfL_3),\roboCurvature,p)}$ is equal to the LoC of original regular cuboid robot with half-lenghts $(\halfL_1,\halfL_2,\halfL_3)$. 
\end{corollary}
\begin{proof}
The proof is immediate from Theorem~\ref{thm:polar_Lp_3D} and Proposition~\ref{prop:bend_rectangular}, as it was shown that each $X-Y$ section for a given $Z$ level of the approximated robot is equivalent to the approximated bent rectangle.
\end{proof}




%
%
\section{Collision Avoidance Constraints}
\label{sec:col_avoid}

In this section, the safety conditions for regular cuboid and bent cuboid in (\ref{eqn:problem_statement_inequality}) are analyzed, and transformed into a set of equality and inequality constraints using the weighted $L_p$ norm, and the weighted polar $L_p$ function proposed in the previous section.

\subsection{Two stage optimization}

The geometry explored in the previous sections will serve to convert the
safety constraint expressed as a set inclusion in
(\ref{eqn:problem_statement_inequality}) into several set exclusion inequality
constraints to be met simultaneously.  In doing so the equivalent
formulation of the constraints will involve the solution to an optimization
problem.  Suppose that the surface of robot is modeled as in
(\ref{eqn:regular_cuboid}) for a regular cuboid, and
(\ref{eqn:kappa_level_3D}) for bent cuboid with some even number $\rP$. In
addition, the surface of $j$-th rectangular obstacle is modeled with
$1$-level set of $\halfL^j$-weighted $L_p$ norm with some even number $\oP$.
Given that analytic surface model of the robot, the safey constraint between
a robot and a single obstacle is
\begin{align} \label{eqn:two_stage_inequality}
  &\left| \left| [\sPose^j]^{-1}\circ \applyG{\rPose_t}{\overline{v}^j}\right| \right|_{\halfL^j,p_o}>1 \\ &\text{such that } 
  \ \
  \label{eqn:two_stage_opt}
  \overline{v}^j:=\arg \min_{v^j\in\mathbb{R}^3} \left| \left| [\sPose^j]^{-1}\circ \applyG{\rPose_t}{v^j}\right| \right|_{\halfL^j,\oP} 
  \\ &\text{subject to} 
  \ \
  \label{eqn:two_stage_equality_regular}
  \left| \left| {v^j}\right| \right|_{\halfL,\rP}=1, \\ &\text{or subject to }
  \ \
  \label{eqn:two_stage_equality_bent}
  \Psi_{(\halfL,\roboCurvature,\rP)}(v^j)=|\roboCurvature|, 
\end{align}
following the notation of \S\ref{sec:prob}. The closest point
$\overline{v}^j$ in (\ref{eqn:two_stage_opt}) is on the surface of the robot
given by (\ref{eqn:two_stage_equality_regular}) for the regular cuboid or
(\ref{eqn:two_stage_equality_bent}) for the bent cuboid in the robot frame. 
This closest point minimizes the $\halfL^j$-weighted $L_p$ obstacle distance
in the obstacle frame. Therefore, the safety constraint is equivalent to
(\ref{eqn:two_stage_inequality}), since the surface of the obstacle is given
by the $1$-level set of $\halfL^j$-weighted $L_p$ norm. This formulation is
only valid when the closest point, $\overline{v}^j$ is available. The
subsequent sections identify necessary and sufficient conditions for the
closest point $\overline{v}^j$ optimization problem.


\subsection{Collision Avoidance Constraint}


\begin{figure}[t!]
  \captionsetup[subfigure]{}
  \centering
  \hfill
  \subfloat[Regular cuboid.]
    {\includegraphics[width=1.15in, clip=true,trim=3.5in 1.2in 6.5in 2.2in]{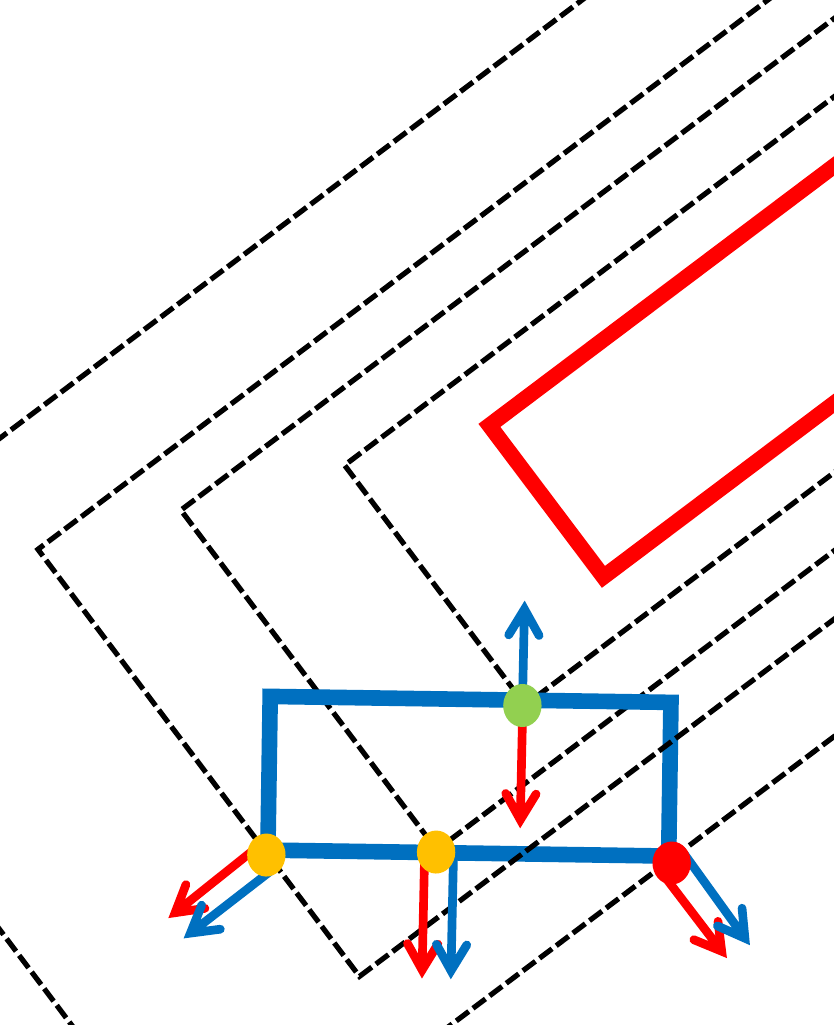}\label{fig:sufficient_regular}}
  \subfloat[Case 1 : Bent cuboid]
    {\includegraphics[width=1.15in, clip=true,trim=3.5in 1.2in 6.5in 2.2in]{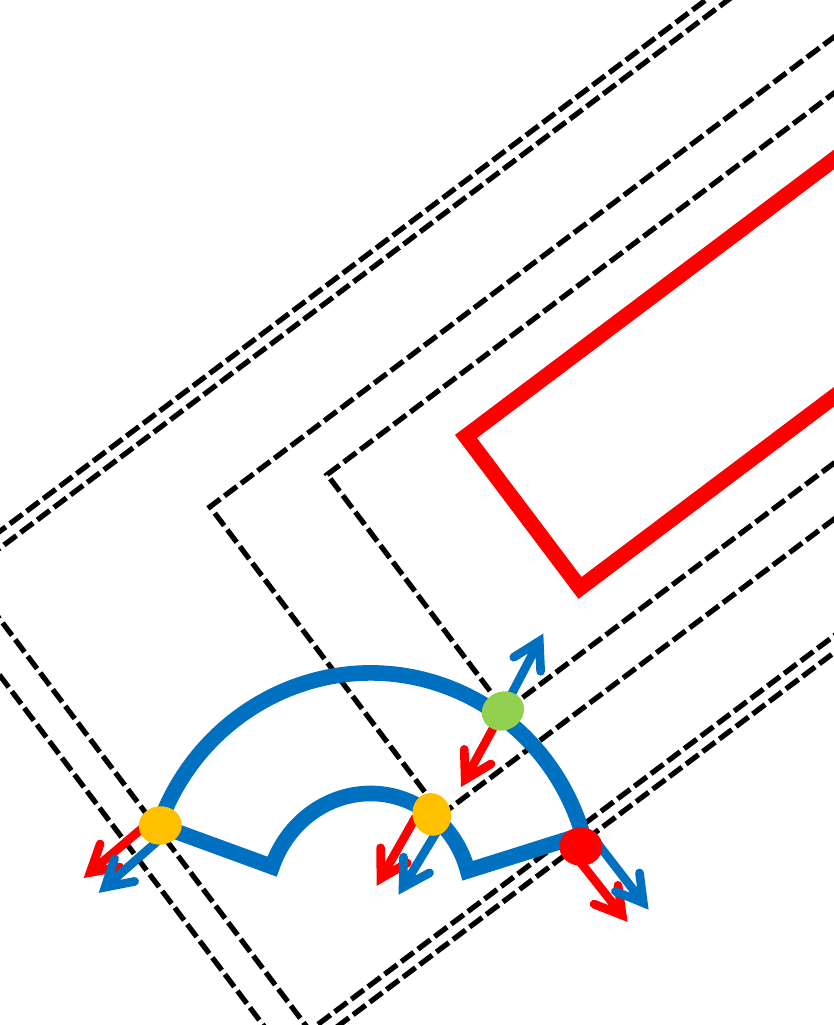}\label{fig:sufficient_bent_easy}}
  \subfloat[Case 2 : Bent cuboid]
    {\includegraphics[width=1.15in, clip=true,trim=3.5in 1.2in 6.5in 2.2in]{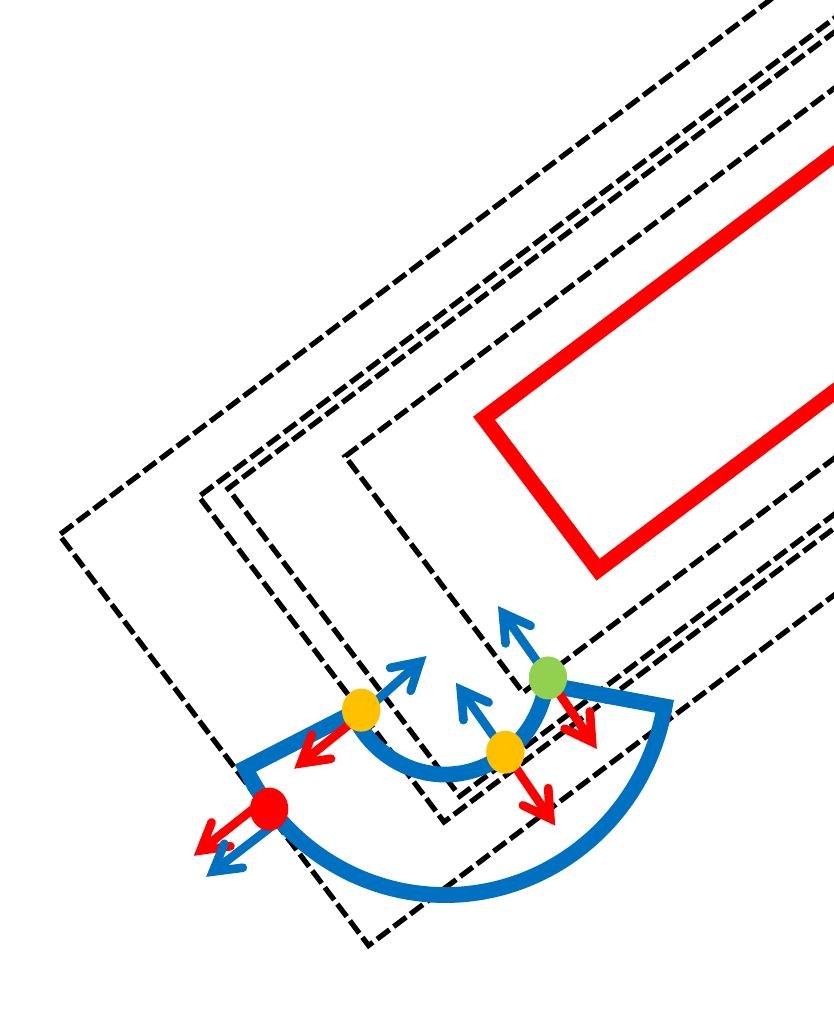}\label{fig:sufficient_bent_hard}}
  \hfill
	\caption{Geometric interpretation of stationary points.}
\label{fig:sufficient}
\end{figure}

\subsubsection{Regular Cuboid Robots}

First, we identify the necessary condition for $\vPose:=\vCoordsSEThree$ to
be a stationary solution to (\ref{eqn:two_stage_opt}) subject to
(\ref{eqn:two_stage_equality_regular}). Let
$w = (w_x,w_y,w_z):=[\sPose^j]^{-1}\circ \applyG{\rPose_t}{\vPose}$  be 
transformation of $\vPose$ to the $j$-the obstacle frame. The gradient of
the cost function in (\ref{eqn:two_stage_opt}) with respect to
$w$, and the gradient of the equality constraint in
(\ref{eqn:two_stage_equality_regular}) with respect to $\vPose$ are 
\begin{eqnarray} \label{eqn:W_regular}
  \mathcal{W}_{\vPose}
    := (w_x^{p-1}/(\halfL^j_1)^p,w_y^{p-1}/(\halfL^j_2)^p,
                                 w_z^{p-1}/(\halfL^j_3)^p)^T
  \\
  \label{eqn:V_regular}
  \mathcal{V}_{\vPose}
    := (\vCoordX^{p-1}/(\halfL_1)^p,\vCoordY^{p-1}/(\halfL_2)^p,
                                    \vCoordZ^{p-1}/(\halfL_3)^p)^T.
\end{eqnarray} 
\begin{theorem} \label{thm:necessary_regular}
The first order necessary condition for $\overline{\vPose}$ to be minimizer of (\ref{eqn:two_stage_opt}) is
\begin{equation} \label{eqn:necessary_regular}
  P_{\overline{\vPose}}([{\rPose_t}]^{-1}\circ \sPose^j({\mathcal{W}_{\overline{\vPose}}})) = 0,
\end{equation}
where 
\begin{equation} \label{eqn:P_regular}
  P_{\overline{\vPose}} 
    = I_{3\times 3} - \mathcal{V}_{\overline{\vPose}}\overline{\vPose}^T,
\end{equation}
with $I_{3\times 3}$ an identity matrix.
\end{theorem}
\begin{proof}
See APPENDIX~\ref{app:necessary_regular}
\end{proof}
The structure of $P_{\overline{\vPose}}$ provides a geometric interpretation
for the necessary condition. 
\begin{proposition} \label{prop:projection_regular}
  $P_{\overline{\vPose}}$ is a projection matrix.
\end{proposition}
\begin{proof}
It suffices to show that $P_{\overline{\vPose}}$ is idempotent since $P$ is
a linear transformation. By computing $P_{\overline{\vPose}}^2$, the
following holds,
\begin{multline} \nonumber
  P_{\overline{\vPose}}^2
    = (I_{3\times 3}-\mathcal{V}_{\overline{\vPose}}\overline{\vPose}^T)^2 \\
    = I_{3\times 3}-2\mathcal{V}_{\overline{\vPose}}\overline{\vPose}^T
      + \mathcal{V}_{\overline{\vPose}}\overline{\vPose}^T
        \mathcal{V}_{\overline{\vPose}}\overline{\vPose}^T 
    = P_{\overline{\vPose}},
\end{multline}
since $\overline{\vPose}^T\mathcal{V}_{\overline{\vPose}}=1$ by (\ref{eqn:two_stage_equality_regular}).
\end{proof}
The next corollary investigatess the range space of $P_{\overline{\vPose}}$.
\begin{corollary}
\label{cor:range_P}
The dimension of the range space of $P_{\overline{\vPose}}$, is $2$.
\end{corollary}
\begin{proof}
Since $P_{\overline{\vPose}}$ is idempotent, each eigenvalue of $P_{\overline{\vPose}}$ is either $0$ or $1$ \cite{horn2012matrix}. The trace of $P_{\overline{\vPose}}$ is given by,
\begin{equation*}
Tr(P_{\overline{\vPose}})=3-\overline{\vPose}^T\mathcal{V}_{\overline{\vPose}}=2.
\end{equation*}
Hence, there exist two eigenvalues of $1$. 
\end{proof}
The next theorem follows immediately.
\begin{theorem}{(Geometric interpretation)}
\label{thm:geometry_regular}
The necessary condition in (\ref{eqn:necessary_regular}) holds for $\vPose$ if and only if the vector $([{\rPose_t}]^{-1}\circ \sPose^j({\mathcal{W}}))$ is orthogonal to the two dimensional tangent space of the boundary of the robot, $\partial B_{((\halfL_1,\halfL_2,\halfL_3),p)}$, at $\vPose$.
\end{theorem}
\begin{proof}
Since $\mathcal{V}_{\overline{\vPose}}$ is a gradient vector, it is orthogonal to the tangent space of $\partial B_{((\halfL_1,\halfL_2,\halfL_3),p)}$, at $\vPose$. In addition, observe that $\mathcal{V}_{\overline{\vPose}}$ is in the nullspace of $P_{\overline{\vPose}}$. Therefore, the row space of $P_{\overline{\vPose}}$ represents the two dimensional tangent space by Corollary~\ref{cor:range_P}. Hence, the necessary condition in (\ref{eqn:necessary_regular}) holds if and only if the vector $([{\rPose_t}]^{-1}\circ \sPose^j({\mathcal{W}}_{\overline{\vPose}}))$ is aligned with $\mathcal{V}_{\overline{\vPose}}$.
\end{proof}


Instead of examining the Hessian of the Hamiltonian in (\ref{eqn:Hamiltonian}), the following proposition provides a sufficient condition for global minima of (\ref{eqn:two_stage_opt}).
\begin{proposition}
\label{prop:sufficient_regular}
A sufficient condition for $\overline{\vPose}$ to be a global minimum is 
\begin{equation} \label{eqn:sufficient_regular}
  ([{\rPose_t}]^{-1}\circ \sPose^j({\mathcal{W}}_{\overline{\vPose}}))^T\overline{\vPose} < 0.
\end{equation}
\end{proposition}
\begin{proof}
First, observe that the regular cuboid robot, $B_{((\halfL_1,\halfL_2,\halfL_3),p)}$, is convex, and the $(\halfL^j,p)$-weighted $L_p$ norm is a convex function. The necessary condition in (\ref{eqn:necessary_regular}) shows that $\mathcal{V}_{\overline{\vPose}}$ and $([{\rPose_t}]^{-1}\circ \sPose^j({\mathcal{W}}_{\overline{\vPose}}))$ are aligned. If these two vectors point in the same direction, in other words $([{\rPose_t}]^{-1}\circ \sPose^j({\mathcal{W}}_{\overline{\vPose}}))^T\mathcal{V}_{\overline{\vPose}}>0$, then there is a point interior to $B_{((\halfL_1,\halfL_2,\halfL_3),p)}$ for which the cost is smaller. Since the level curves of $(\halfL^j,p)$-weighted $L_p$ norm are continuous, there exist a point on the boundary $\partial B_{((\halfL_1,\halfL_2,\halfL_3),p)}$ with smaller cost. Therefore, $([{\rPose_t}]^{-1}\circ \sPose^j({\mathcal{W}}_{\overline{\vPose}}))^T\mathcal{V}_{\overline{\vPose}}<0$ should hold for a global minimum point. 

By using (\ref{eqn:Hamiltonian_grad0}) in APPENDIX \ref{app:necessary_regular}, the inner product between two normal vectors is explicitly computed as
\begin{equation*}
([{\rPose_t}]^{-1}\circ \sPose^j({\mathcal{W}}_{\overline{\vPose}}))^T\mathcal{V}_{\overline{\vPose}}=([{\rPose_t}]^{-1}\circ \sPose^j({\mathcal{W}}_{\overline{\vPose}}))^T\vPose||\mathcal{V}_{\overline{\vPose}}||_2.
\end{equation*}
Since $||\mathcal{V}_{\overline{\vPose}}||_2$ is not equal to zero for every $\vPose$ on the surface of the robot, (\ref{eqn:sufficient_regular}) only holds at the minimum.
\end{proof}

Interestingly, the inequality constraint in (\ref{eqn:sufficient_regular}) is equivalent to $\lambda>0$, where $\lambda$ is the Lagrange multiplier in the Hamiltonian (\ref{eqn:Hamiltonian}). This condition is similar to the Karush-Kuhn-Tucker (KKT) condition in the sense that the sign of the Lagrange multiplier must be non-negative.

A geometric explanation of Theorem~\ref{thm:geometry_regular} and
Proposition~\ref{prop:sufficient_regular} is demonstrated in
Figure~\ref{fig:sufficient_regular}. The example depicted is in planar space
in order to better explain the direction of the normal vectors of the
surface. The blue weighted $L_p$ rectangle represents the robot's boundary,
and the red weighted $L_p$ rectangle represents the obstacle's boundary. The
figure is drawn in the robot frame. The black dashed lines represent the
level sets of the obstacle's weighted $L_p$ norm, where the level value gets
smaller closer to the red rectangle. The blue arrows represent the outward
normal to the surface of the robot, and the red arrows represent the outward
normal to the level curves of the obstacle's weighted $L_p$ norm. In this
example, there exist four points, two orange circles, one green circle, and
one red circle on the robot's boundary. The green circle is the
desired closest point, and the red circle is the maximum distance in terms
of obstacle's weighted $L_p$ norm. All four circles are stationary solutions
satisfying the necessary condition in (\ref{eqn:necessary_regular}). In all
cases, the normal vectors are aligned. The blue arrow and the red arrow are 
pointing in the same direction for all three circles except for the green
circle, which is the desired closest point. This graphically verifies the
Proposition~\ref{prop:sufficient_regular}.


Therefore, the original optimization problem is replaced by introducing an auxiliary closest point $\overline{\vPose}$ which satisfies four equality conditions, (\ref{eqn:necessary_regular}) and (\ref{eqn:two_stage_equality_regular}), and one inequality condition (\ref{eqn:sufficient_regular}). By combining these constraints with inequality condition in (\ref{eqn:two_stage_inequality}), the safety condition for the regular cuboid can be given by four equality constraints and two inequality constraints. 

\subsubsection{Bent Cuboid Robots}
Here, the necessary condition for $\vPose$ to be a stationary solution to (\ref{eqn:two_stage_opt}) subject to (\ref{eqn:two_stage_equality_bent}) is considered. Let $w = (w_x,w_y,w_z):=[\sPose^j]^{-1}\circ \applyG{\rPose_t}{\vPose}$  be the transformation of $\vPose$ to the $j$-th obstacle frame. The gradient of the cost function in (\ref{eqn:two_stage_opt}) with respect to $w$ is the same as in (\ref{eqn:W_regular}), and the gradient of the equality constraint in (\ref{eqn:two_stage_equality_bent}) with respect to $\vPose$ is
\begin{equation} \label{eqn:V_polar}
  \mathcal{U}_{\vPose} := \mathcal{V}_{\vPose}^{T}\overline{Q}^{-1},
\end{equation} 
where 
\begin{eqnarray} \label{eqn:VSE3}
  \mathcal{V}_{\vPose} &:=& 
    \begin{pmatrix}
      ({R_{\mathcal{T}_\roboCurvature}(\vPose)-1})^{p-1} / (\halfL_2)^p \\
      (\theta_{\mathcal{T}_\roboCurvature}(\vPose)-\theta_0)^{p-1}/(\halfL_1)^p
      \\
      (|\roboCurvature|\vCoordZ)^{p-1}/(\halfL_3)^p)^T
    \end{pmatrix} \\
  \label{eqn:QSE3}
  \overline{Q} &:=& 
    \begin{pmatrix}
      Q & \textbf{0} \\
      \textbf{0}^T & 1
    \end{pmatrix}
\end{eqnarray}
such that $\textbf{0}$ is a zero vector in $\mathbb{R}^2$, and 
\begin{equation} \label{eqn:Q}
  Q = \begin{pmatrix}
      \cos(\theta_{\mathcal{T}_\roboCurvature}(\vPose)) & 
      R_{\mathcal{T}_\roboCurvature}(\vPose)
        \sin(\theta_{\mathcal{T}_\roboCurvature}(\vPose)) \\
      \sin(\theta_{\mathcal{T}_\roboCurvature}(\vPose)) & 
      -R_{\mathcal{T}_\roboCurvature}(\vPose)
        \cos(\theta_{\mathcal{T}_\roboCurvature}(\vPose))
  \end{pmatrix}.
\end{equation}
Now, defined $\nu$ as follows,
\begin{equation*}
\nu:= (R_{\mathcal{T}_\roboCurvature}(\vPose)-1, \theta_{\mathcal{T}_\roboCurvature}(\vPose)-\theta_0,|\roboCurvature|\vCoordZ)^T.
\end{equation*}
Using the equality constraint in (\ref{eqn:two_stage_equality_bent}),  the following holds,
\begin{equation}
\label{eqn:equality_psi}
\nu^T\mathcal{V}_{\vPose}=|\roboCurvature|^p.
\end{equation}
\begin{theorem}
\label{thm:necessary_bent}
The first order necessary condition for $\overline{\vPose}$ to be minimizer of (\ref{eqn:two_stage_opt}) is
\begin{equation}
\label{eqn:necessary_bent}
P_{\overline{\vPose}}([{\rPose_t}]^{-1}\circ \sPose^j({\mathcal{W}}_{\overline{\vPose}}))=0,
\end{equation}
where 
\begin{equation}
\label{eqn:P_bent}
P_{\overline{\vPose}}=I_{3\times 3}-(1/|\roboCurvature|^p)\overline{Q}^{-T}\mathcal{V}_{\overline{\vPose}}\nu^T\overline{Q}^T,
\end{equation}
with $\overline{Q}^{-T}$ as an inverse transpose of $\overline{Q}$.
\end{theorem}
\begin{proof}
See APPENDIX~\ref{app:necessary_bent}
\end{proof}
Similar to the regular cuboid case, $P_{\overline{\vPose}}$ is also a projection matrix, which provides the geometric interpretation of the necessary condition.
\begin{proposition}
$P_{\overline{\vPose}}$ is a projection.
\end{proposition}
\begin{proof}
Follows the proof of Proposition~\ref{prop:projection_regular} but using (\ref{eqn:equality_psi}). 
\end{proof}
\begin{corollary}
\label{cor:range_P_bent}
The dimension of the range space of $P_{\overline{\vPose}}$ is 2.
\end{corollary}
\begin{proof}
Similar to the proof of Corollary~\ref{cor:range_P}, it follows by explicitly computing the trace, and using the fact that $P_{\overline{\vPose}}$ is idempotent. 
\end{proof}
Furthermore, $\mathcal{V}_{\overline{\vPose}}$ is an outward normal vector to the level set of the polar coordinate $(R_{\mathcal{T}_\roboCurvature}(\vPose), \theta_{\mathcal{T}_\roboCurvature}(\vPose))$, followed by the coordinate transformation explained in \S\ref{sec:Surface:Bend_cuboid}. The interchange between the polar coordinate and the original coordinate in the robot's frame is governed by $Q^{-T}$ operator. 

\textbf{Remark 3.}
A similar matrix to $Q$ or $Q^{-T}$ appears in the coordinate transformations of velocity vectors as a near-identity diffeomorphism (NID) \cite{OSaber2002NID}. Although the usage of NID is different from here, the $Q$ matrix has useful properties such as that $Q^{T}$ is almost equivalent to the $Q^{-1}$.  It is easy to check that $\det{Q}=-R_{\mathcal{T}_\roboCurvature}(\vPose)$, and the inverse, $Q^{-1}$, is
\begin{equation*}
\begin{pmatrix}
\cos(\theta_{\mathcal{T}_\roboCurvature}(\vPose)) & \sin(\theta_{\mathcal{T}_\roboCurvature}(\vPose))
\\
(1/R_{\mathcal{T}_\roboCurvature}(\vPose))\sin(\theta_{\mathcal{T}_\roboCurvature}(\vPose)) & -(1/R_{\mathcal{T}_\roboCurvature}(\vPose))\cos(\theta_{\mathcal{T}_\roboCurvature}(\vPose))
\end{pmatrix}.
\end{equation*}
Since $\det{Q}<0$, it is an orientation reversing transformation. 


The geometric interpretation of the necessary condition in (\ref{eqn:necessary_bent}) is similar to the one in the regular cuboid case. 
\begin{theorem}{(Geometric interpretation)}
\label{thm:geometry_bent}
The necessary condition in (\ref{eqn:necessary_bent}) holds for $\vPose$ if and only if the vector $([{\rPose_t}]^{-1}\circ \sPose^j({\mathcal{W}}_{\overline{\vPose}}))$ is orthogonal to the two dimensional tangent space of the boundary of the robot, $\partial B_{((\halfL_1,\halfL_2,\halfL_3),\roboCurvature,p)}$, at $\vPose$.
\end{theorem}
\begin{proof}
Since $\overline{Q}^{-T}\mathcal{V}_{\overline{\vPose}}$ is in the null space of $P_{\overline{\vPose}}$, and it is a gradient vector, then, by using Corollary~\ref{cor:range_P_bent}, the row space of $P_{\overline{\vPose}}$ is found to be equal to the tangent space of the boundary of the robot, $\partial B_{((\halfL_1,\halfL_2,\halfL_3),\roboCurvature,p)}$, at $\vPose$. Therefore, the necessary condition in (\ref{eqn:necessary_bent}) holds if and only if the vector $([{\rPose_t}]^{-1}\circ \sPose^j({\mathcal{W}}_{\overline{\vPose}}))$ is aligned with $\overline{Q}^{-T}\mathcal{V}_{\overline{\vPose}}$.
\end{proof}
The examples of the stationary points are demonstrated in Figure~\ref{fig:sufficient_bent_easy} and figure~\ref{fig:sufficient_bent_hard}, where all four circles in each figure show the stationary points. The green circle is the point with minimum cost, and the red circle is the one with maximum cost. As confirmed in Theorem~\ref{thm:geometry_bent}, the normal vectors to each level sets of robot's boundary and extended obstacle's boundary in robot's frame, $\overline{Q}^{-T}\mathcal{V}_{\overline{\vPose}}$ and $([{\rPose_t}]^{-1}\circ \sPose^j({\mathcal{W}}_{\overline{\vPose}}))$, are aligned at the stationary points. 

However, different from the regular cuboid case, there exist a case when two normal vectors at the non-minimal stationary point are pointing at the opposite direction. In Figure~\ref{fig:sufficient_bent_hard}, all three stationary points have the opposite normal vectors including the green circle. This is due to the nonconvex shape of the robot, and so the Proposition~\ref{prop:sufficient_regular} does not hold for every configuration of the bent robot. 

Nevertheless, checking whether two normal vectors are pointing opposite direction is important since it eliminates the possibility of finding the maximum point (shown as red circles in Figrue~\ref{fig:sufficient_bent_hard}), and it is necessarily true that the globally minimum point should satisfy this property. Furthermore, in many cases as in Figure~\ref{fig:sufficient_bent_easy}, only the globally minimum point attains this property among the other candidates. 

Therefore, similar inequality test to (\ref{eqn:sufficient_regular}) is proposed by using the gradient of the Hamiltonian in (\ref{eqn:Hamiltonian_grad_bent}) and (\ref{eqn:lambda_bent}),
\begin{equation}
\label{eqn:sufficient_bent1}
(1/|\roboCurvature|^{p_r})([{\rPose_t}]^{-1}\circ \sPose^j({\mathcal{W}}_{\overline{\vPose}}))^T\overline{Q}\nu<0,
\end{equation}
which is also can be viewed as the Lagrange multiplier $\lambda$ needs to be positive in the Hamiltonian (\ref{eqn:Hamiltonian_bent}).

In addition, observe that all the faces of the bent cuboid robot, $\partial B_{((\halfL_1,\halfL_2,\halfL_3),\roboCurvature,p)}$, have a convex shape except the face with the shorter arch which has a concave shape. Therefore, if there exist a minimum point on the seven convex faces (including edges), the obstacle weighted $L_p$ norm distance should be smaller than all four corner of the faces. 

As a result, the following eight inequalities are imposed to the candidate $\overline{v}$ in addition to the inequality condition in (\ref{eqn:sufficient_bent1}). All eight corners of bent cuboid in robot's frame are represented by 
\begin{equation*}
\xi^i:=\begin{pmatrix}\mathcal{B}_\roboCurvature(u_i)\\ \halfL_3\end{pmatrix}, \chi^i=\begin{pmatrix}\mathcal{B}_\roboCurvature(u_i)\\ -\halfL_3\end{pmatrix},
\end{equation*}
where $\mathcal{B}_\roboCurvature(u_i)$ is the mapping from the regular cuboid to bent cuboid with curvature, $\roboCurvature$, (See (\ref{eqn:warping_x}) for the definition in Appendix~\ref{app:polar}), and $u_i\in\mathbb{R}^2$ for $i\in\{1,2,3,4\}$ represents the four corners of regular rectangle with half lengths, $(\halfL_1,\halfL_2)$. The eight inequality constraints combined with AND operation are given as
\begin{eqnarray}
\label{eqn:8_constraint_top}
&\land_{i=1}^4 (\left| \left| [\sPose^j]^{-1}\circ \applyG{\rPose_t}{\xi^i}\right| \right|_{\halfL^j,\oP}> \left| \left| [\sPose^j]^{-1}\circ \applyG{\rPose_t}{\overline{\vPose}}\right| \right|_{\halfL^j,\oP})\\
\label{eqn:8_constraint_bottom}
&\land_{i=1}^4 (\left| \left| [\sPose^j]^{-1}\circ \applyG{\rPose_t}{\chi^i}\right| \right|_{\halfL^j,\oP}>\left| \left| [\sPose^j]^{-1}\circ \applyG{\rPose_t}{\overline{\vPose}}\right| \right|_{\halfL^j,\oP}).
\end{eqnarray}
In order to find the closest point, the four equality constraints (necessary condition in (\ref{eqn:necessary_bent}) and (\ref{eqn:two_stage_equality_bent})), and nine inequality constraints (sign of $\lambda$ condition in (\ref{eqn:sufficient_bent1}), and eight constraints for the corners in (\ref{eqn:8_constraint_top}-\ref{eqn:8_constraint_bottom})) are considered in this paper. 

Lastly, by combining these constraints with the inequality condition in (\ref{eqn:two_stage_inequality}), the safety conditions for bent cuboids are given by four equality constraints and ten inequality constraints.

\textbf{Remark 4.}
Although the idea of considering the eight corners is similar to \cite{Hyun2017Lp}, doing so serves a different purpose as the eight corners are not regarded as candidate collision points but used to find the closest point to the obstacle (by avoiding local minima). Therefore, the edge to edge collision is naturally captured by the proposed method, which was not available in \cite{Hyun2017Lp}.

\textbf{Remark 5.}
If the global minimum point occurs in the concave face (excluding the edges), then there exist at most two points, the true minimum point and the one on the edges of the concave face, which satisfy above necessary constraint and the eight inequalities. The only scenario when it could be a problem is the false positive case when the optimizer finds the local minimum on the edge while the obstacle intrudes the robot. However, it does not mean that the worst case always happens, and even it does, one could escape from the local minima by providing an additional inequality constraint for checking if the weighted polar $L_p$ distance from each corner of the obstacle to the CoM of the robot is greater than $|\roboCurvature|$. Nevertheless, the full analysis of this worst case is important, and the authors leave it as a future task.

\section{Optimal Path Planning}
\label{sec:path}

\begin{figure*}[t!]
  \captionsetup[subfigure]{}
  \centering
  \subfloat[Path snapshots]
    {{\includegraphics[width=0.255\textwidth, clip=true,trim=0.3in 0.05in 0.55in 0.3in]{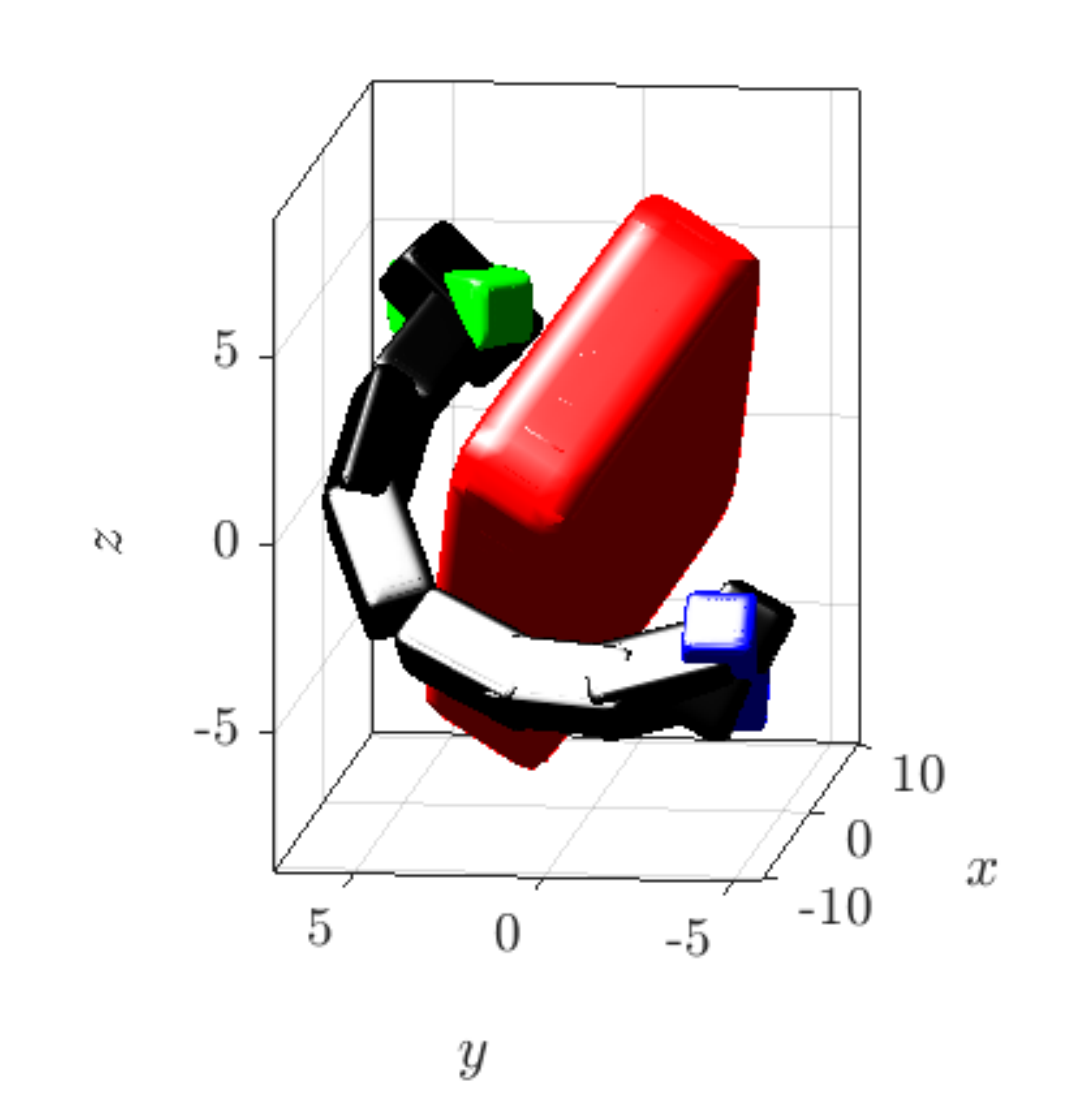}}\label{fig:TRO_figure_cub_cub_Lp}}
  \hspace*{0.01in}
  \subfloat[Normal vectors]
    {\hspace*{0.01in}{\includegraphics[width=0.25\textwidth, clip=true,trim=0.3in 0in 0.7in 0.30in]{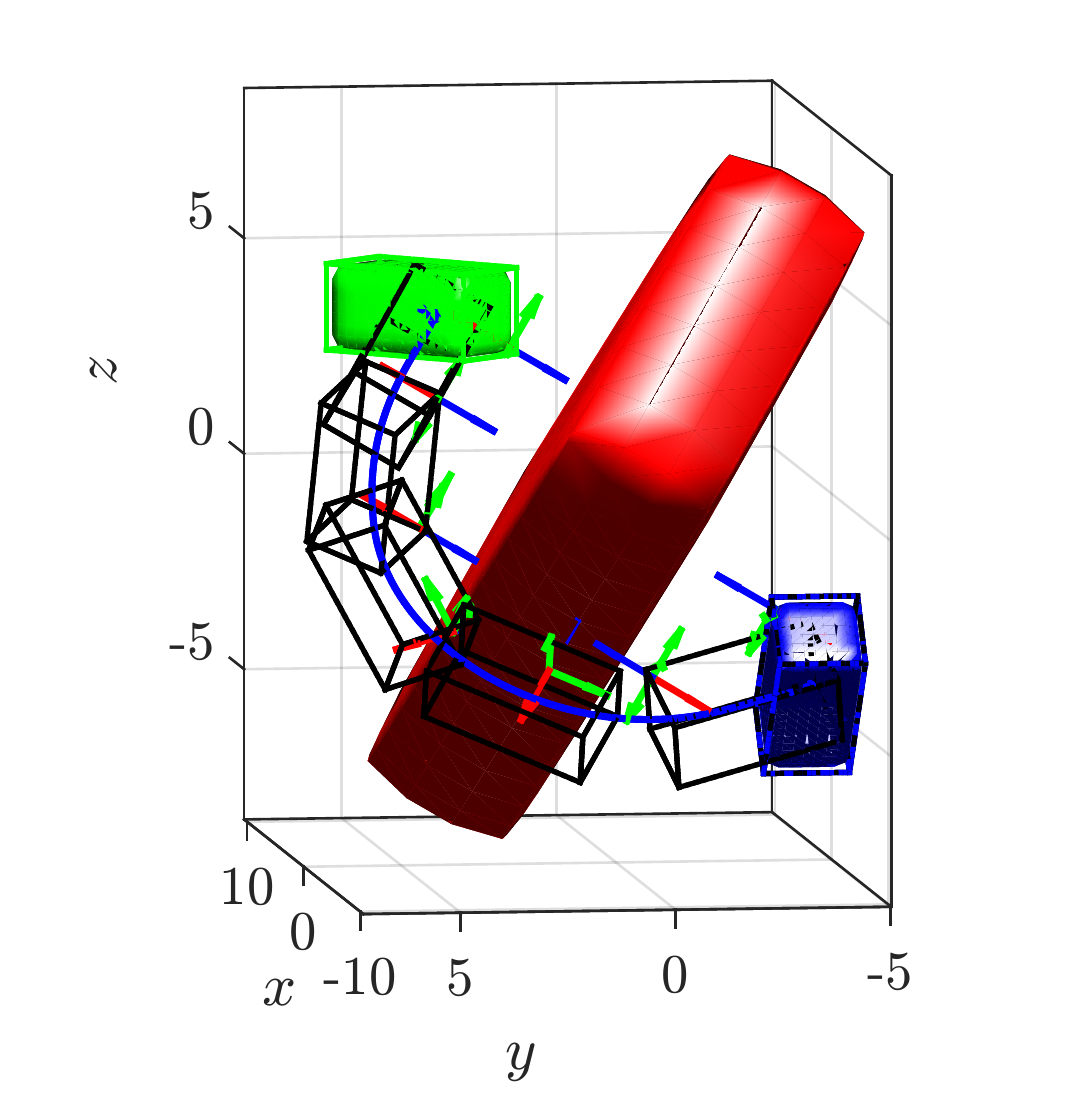}}\label{fig:TRO_figure_cub_cub_Nec_Suff}\hspace*{0.01in}}
  \subfloat[Controls for regular cuboid]{
	\begin{tikzpicture}
    \node[anchor=south east] (w1) at(0.3,-0.2)
    {{\includegraphics[width=0.134\textwidth, clip=true,trim=0.3in 0in 0.8in 0.13in]{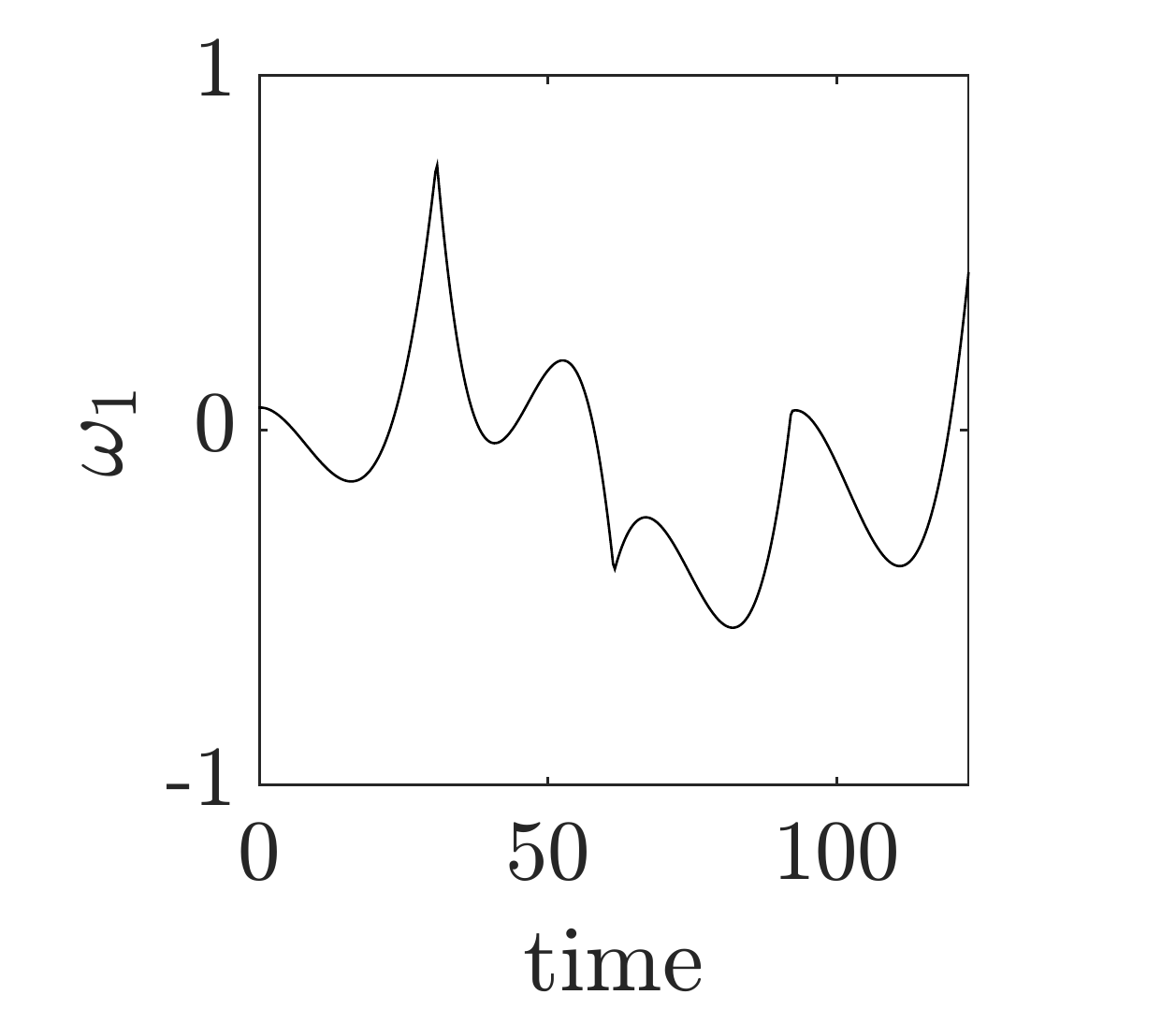}}};
\node[anchor=south] (w2) at (1.5,-0.2)
    {{\includegraphics[width=0.138\textwidth, clip=true,trim=0.1in 0in 0.8in 0.13in]{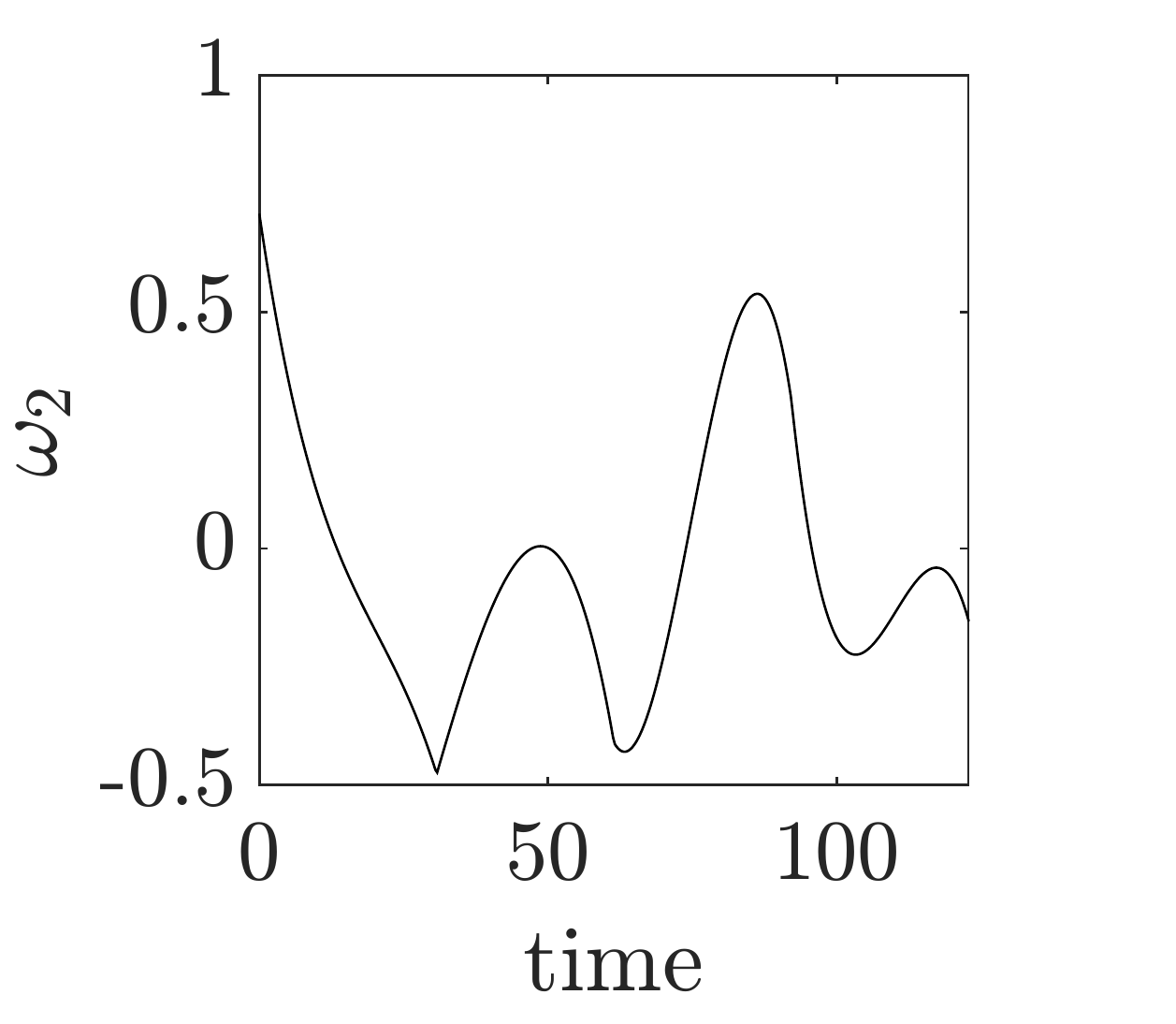}}};
		\node[anchor=south west] (w3) at (2.8,-0.2) 
    {{\includegraphics[width=0.131\textwidth, clip=true,trim=0.4in 0in 0.8in 0.13in]{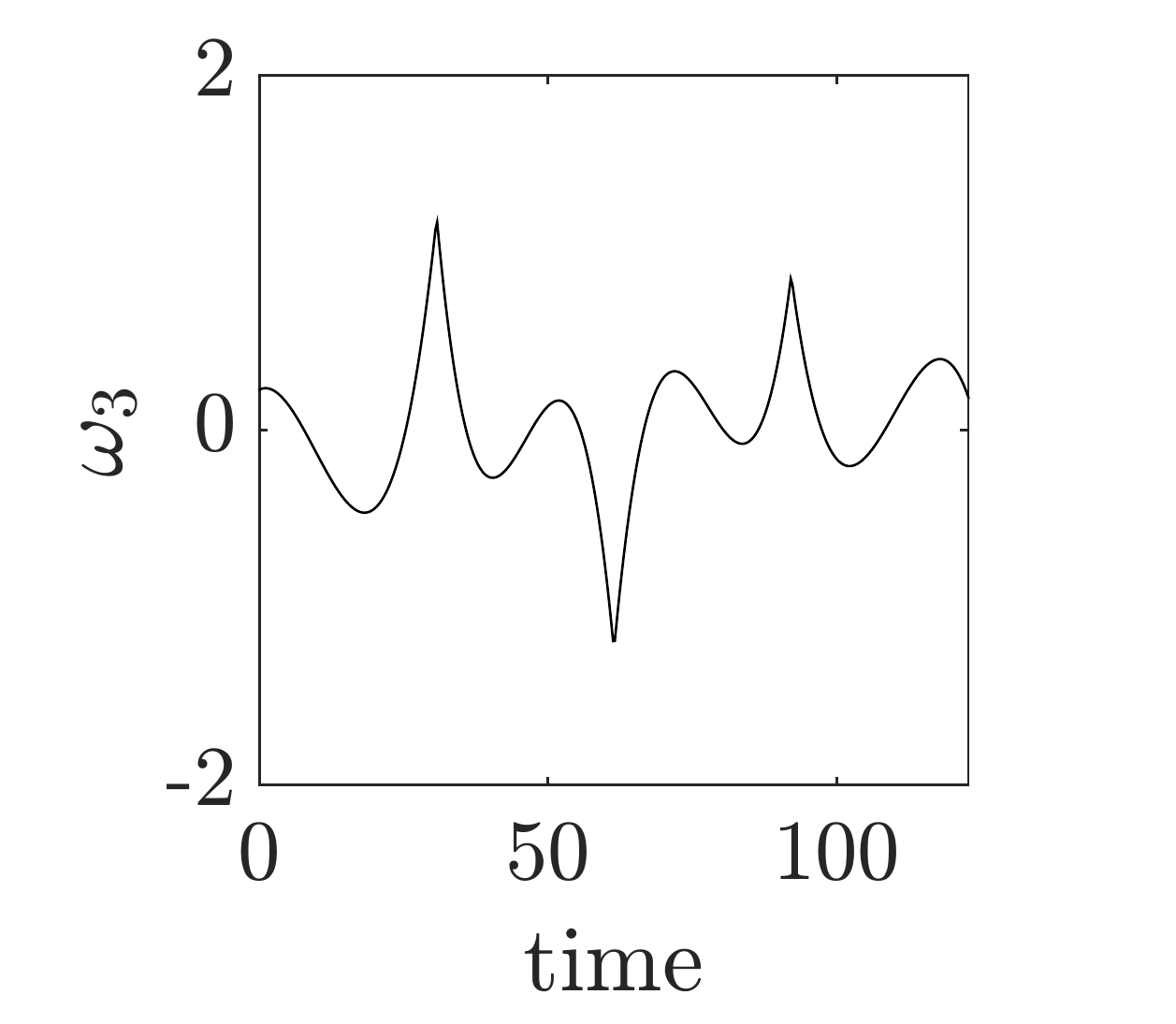}}};
		\node[anchor=north west] (speed.center) at (-0.9,0)
		{{\includegraphics[width=0.232\textwidth, clip=true,trim=0.0in 0.1in 0.6in 0.13in]{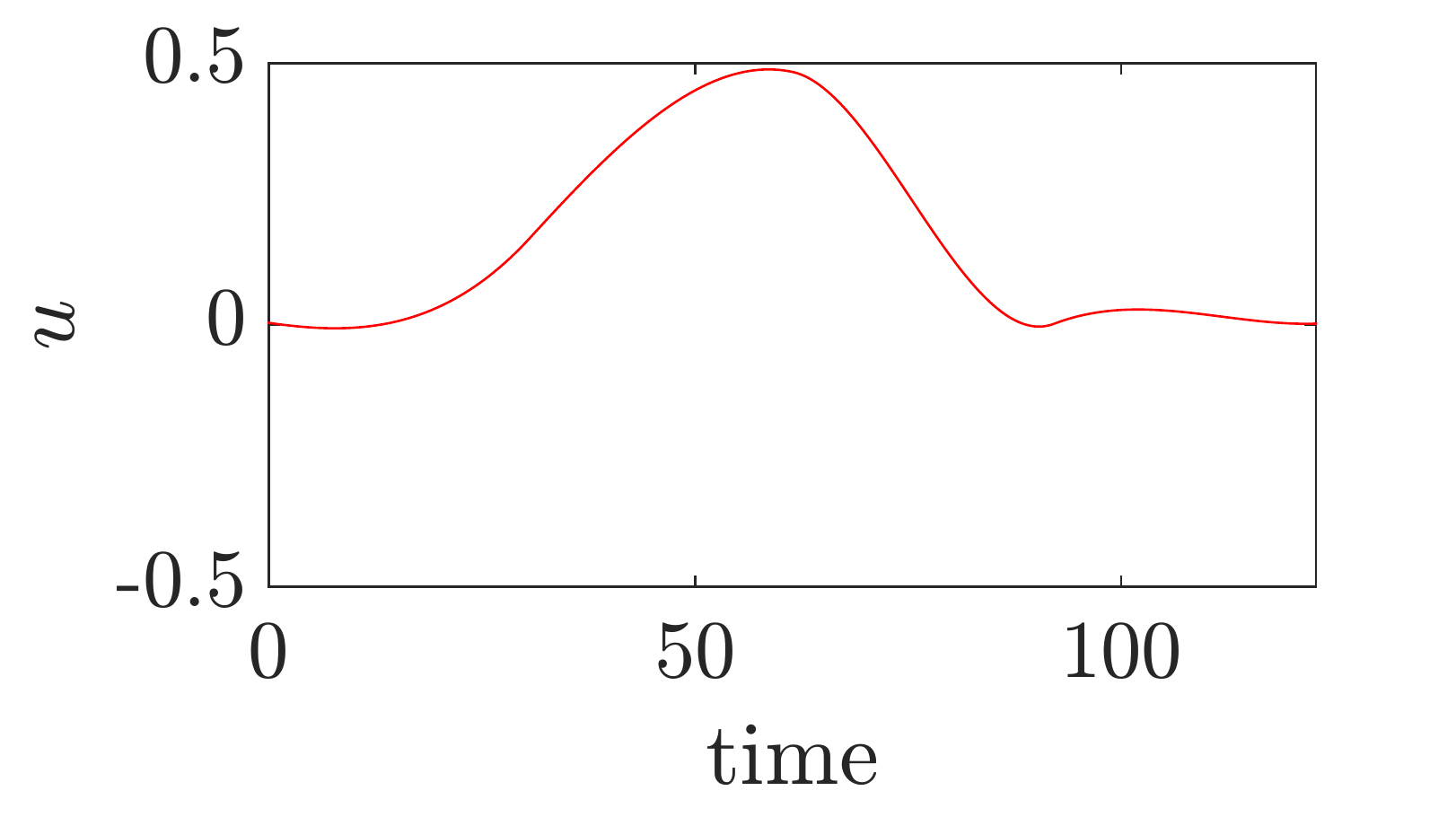}}};
    \end{tikzpicture}\label{fig:TRO_figure_cuboid_control}}
	\caption{Shortest path example for rigid regular cuboid in $SE(3)$}
    \label{fig:cuboid_optimal}
\end{figure*}
\begin{figure*}[t!]
  \captionsetup[subfigure]{}
  \centering
	\subfloat[Path snapshots]
    {{\includegraphics[width=0.255\textwidth, clip=true,trim=0.55in 0in 0.77in 0.3in]{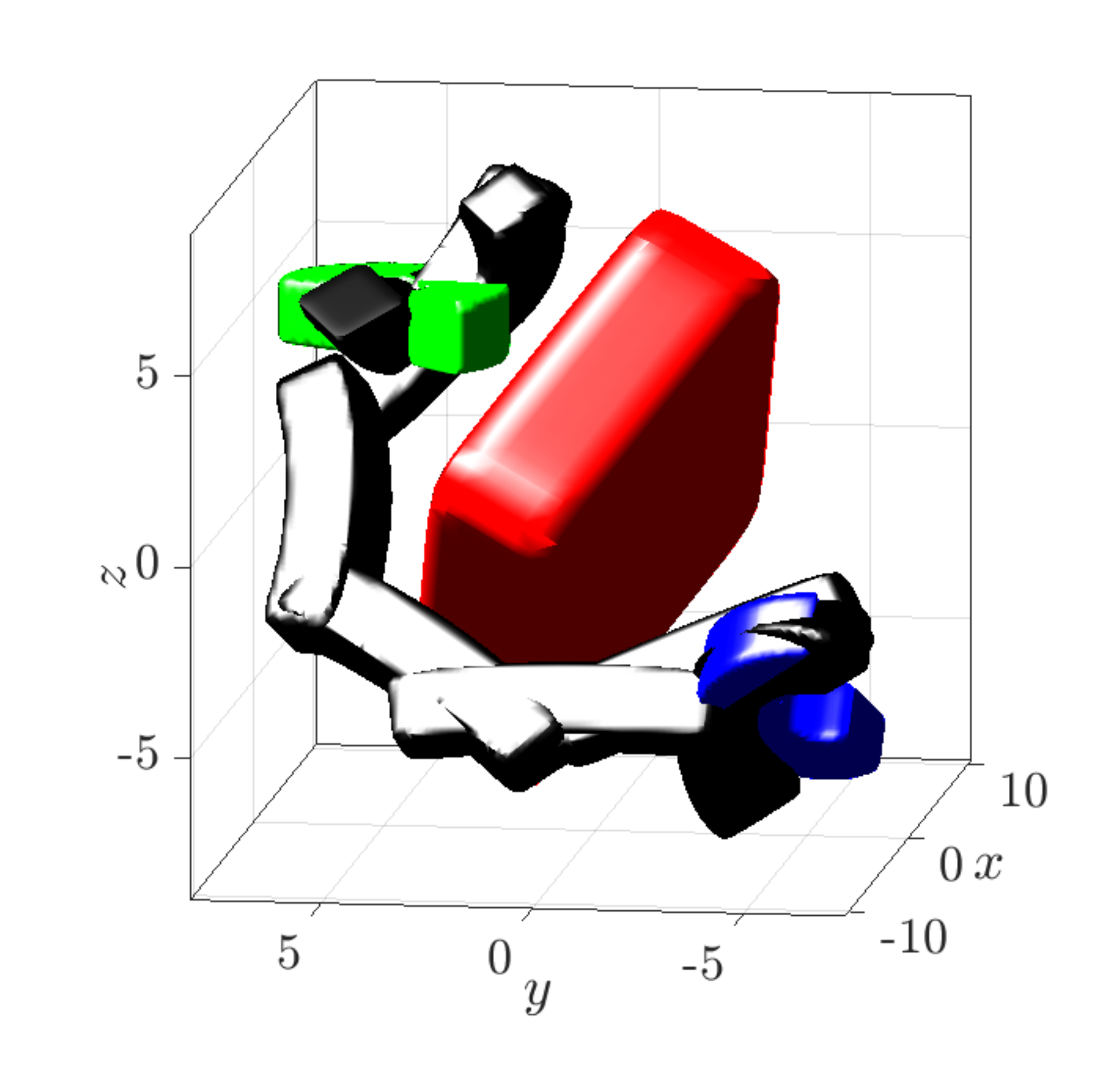}}\label{fig:TRO_figure_bend_cub_Lp}}
	\subfloat[Normal vectors]
	{{\includegraphics[width=0.259\textwidth, clip=true,trim=0.7in 0.1in 1.06in 0.33in]{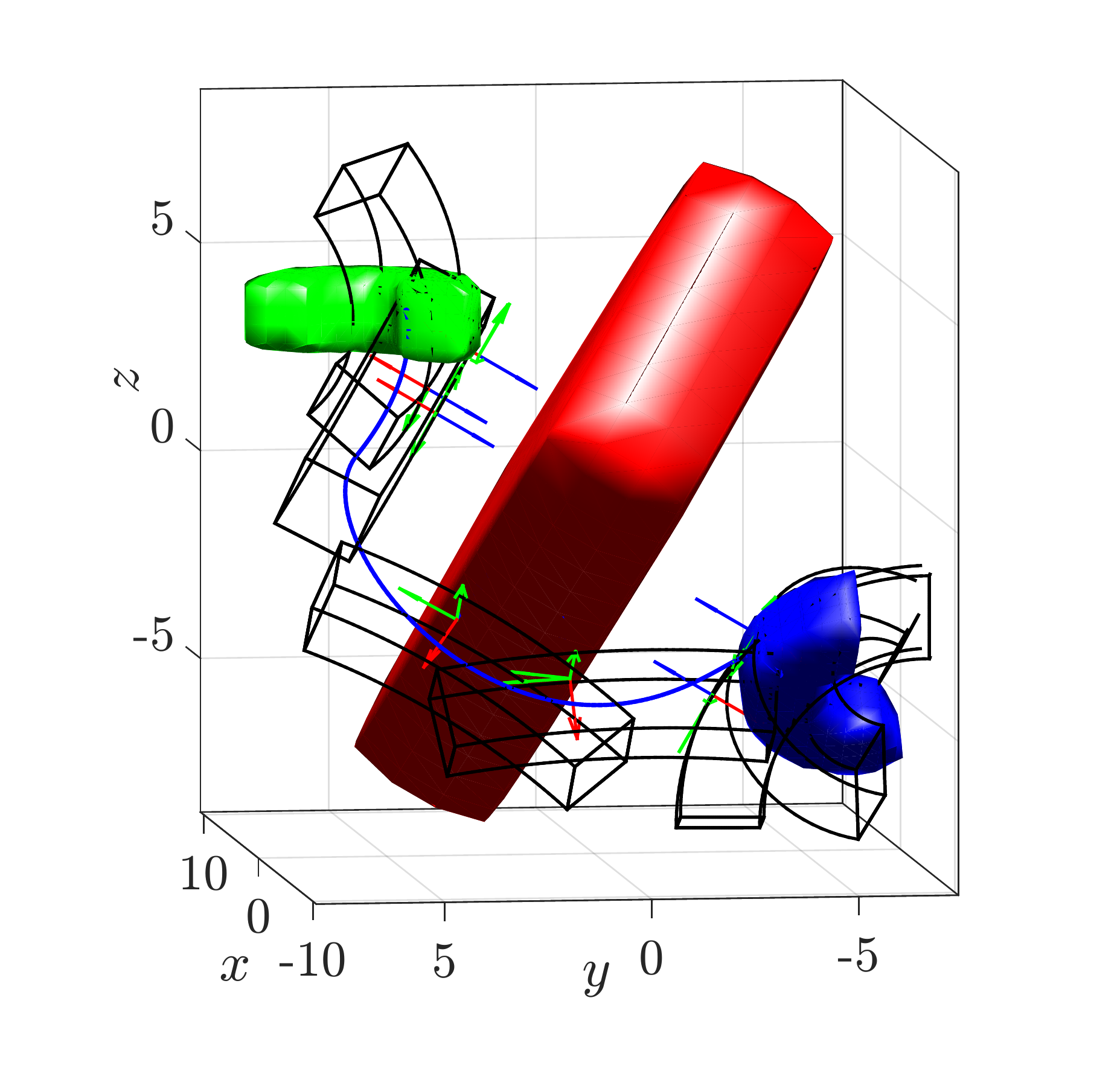}}\label{fig:TRO_figure_bend_cub_Nec_Suff}}
	\subfloat[Controls for bendable cuboid]{
	\begin{tikzpicture}
\node[anchor=south east] (wb1) at (-0.2,-0.2)
    {{\includegraphics[width=0.134\textwidth, clip=true,trim=0.3in 0in 0.8in 0.13in]{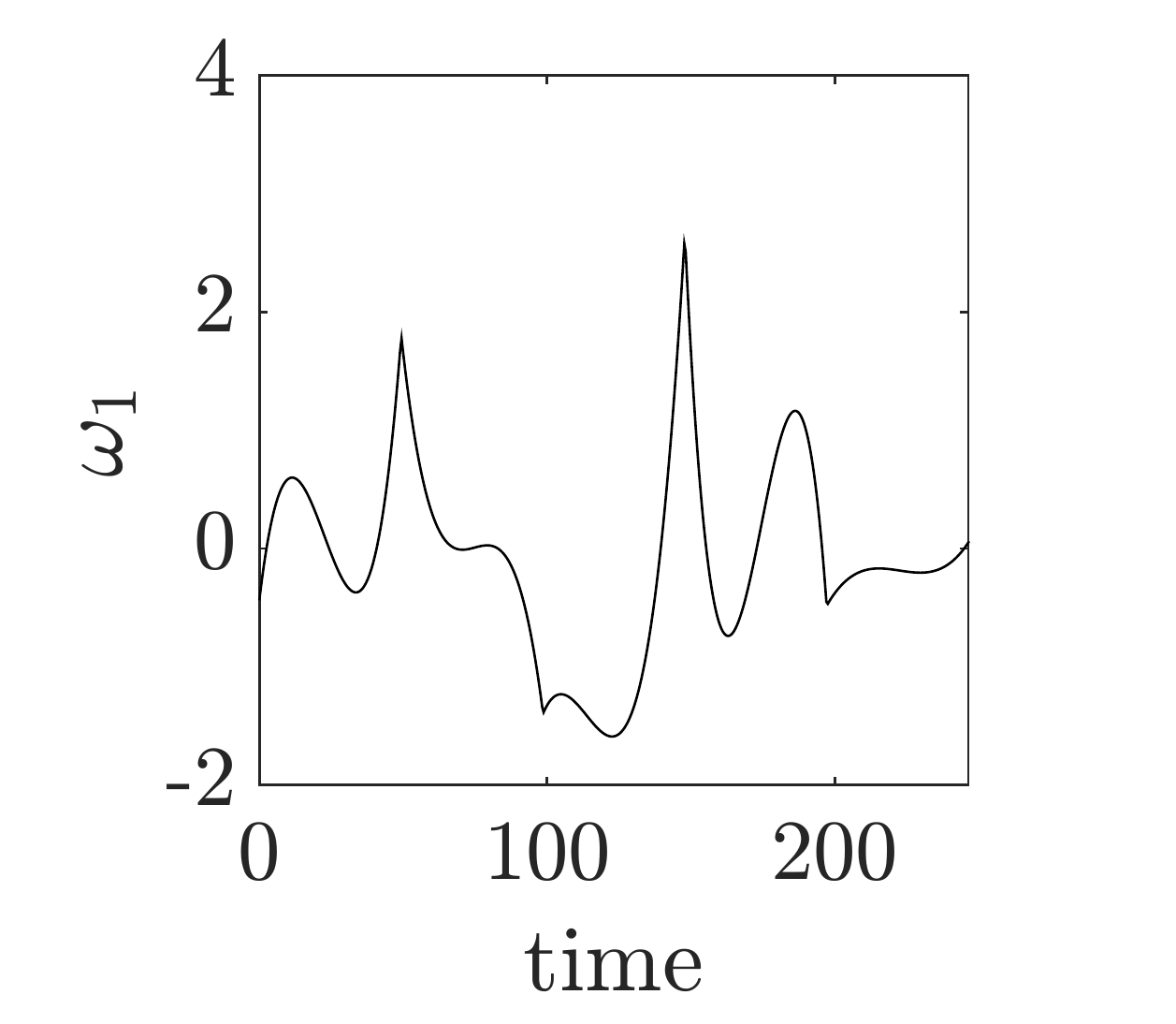}}};
\node[anchor=south ] (wb2) at (1.0,-0.2)
    {{\includegraphics[width=0.138\textwidth, clip=true,trim=0.1in 0in 0.8in 0.13in]{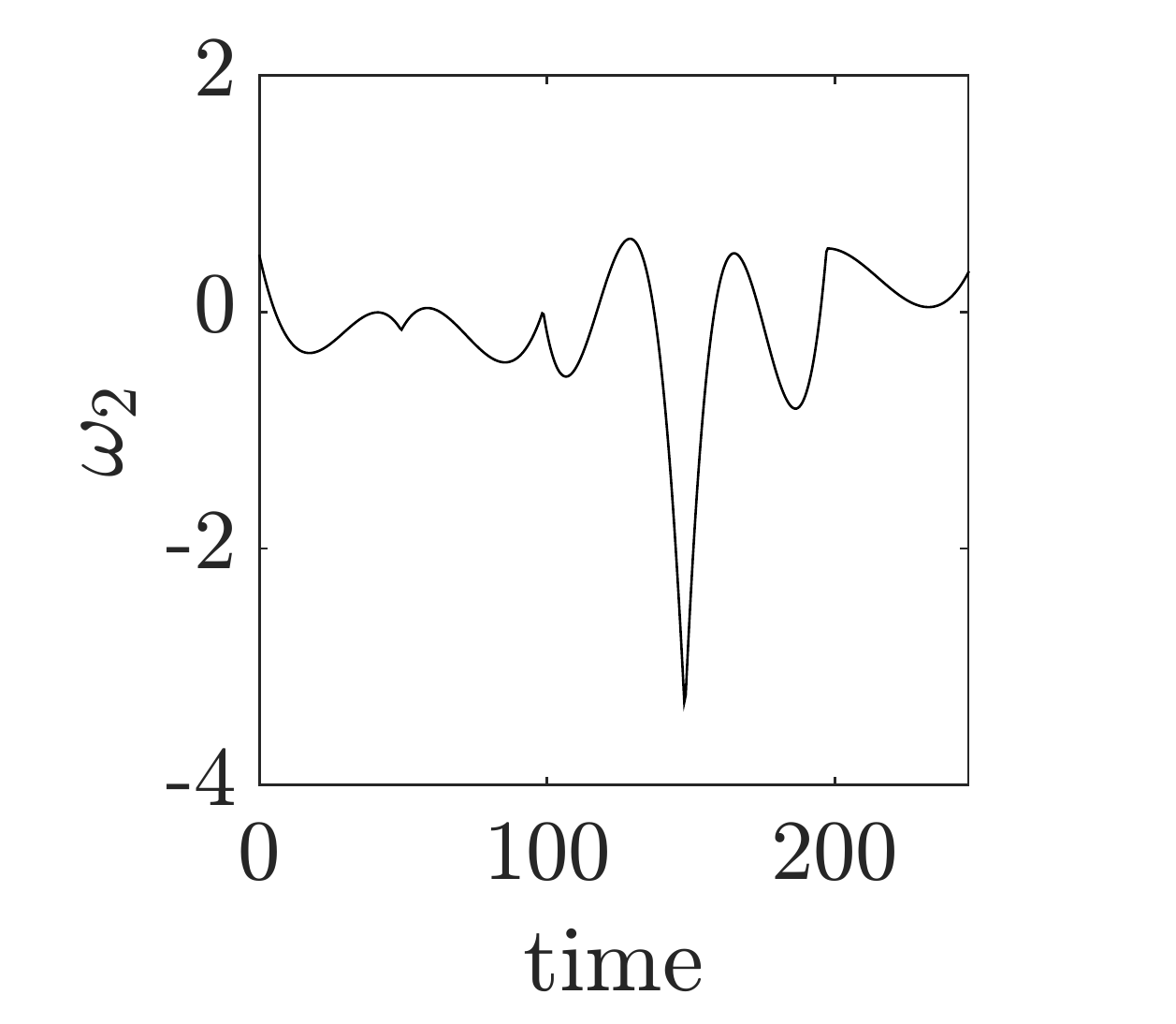}}};
		\node[anchor=south west] (wb3) at (2.2,-0.2) 
    {{\includegraphics[width=0.131\textwidth, clip=true,trim=0.3in 0in 0.8in 0.13in]{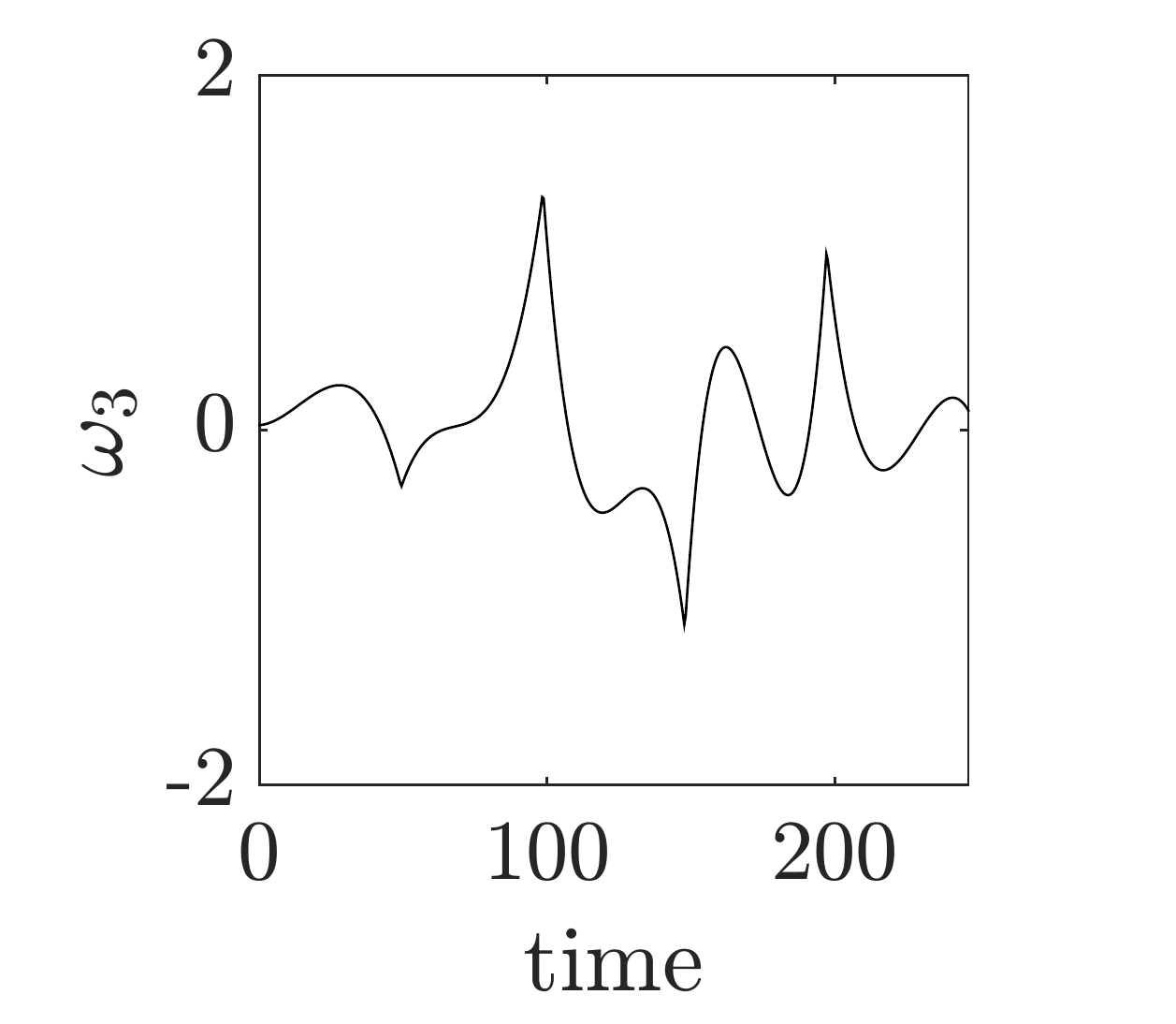}}};
		\node[anchor=north east] (speed.center)  at (1.00,0)
		{{\includegraphics[width=0.202\textwidth, clip=true,trim=0.1in 0.1in 0.55in 0.13in]{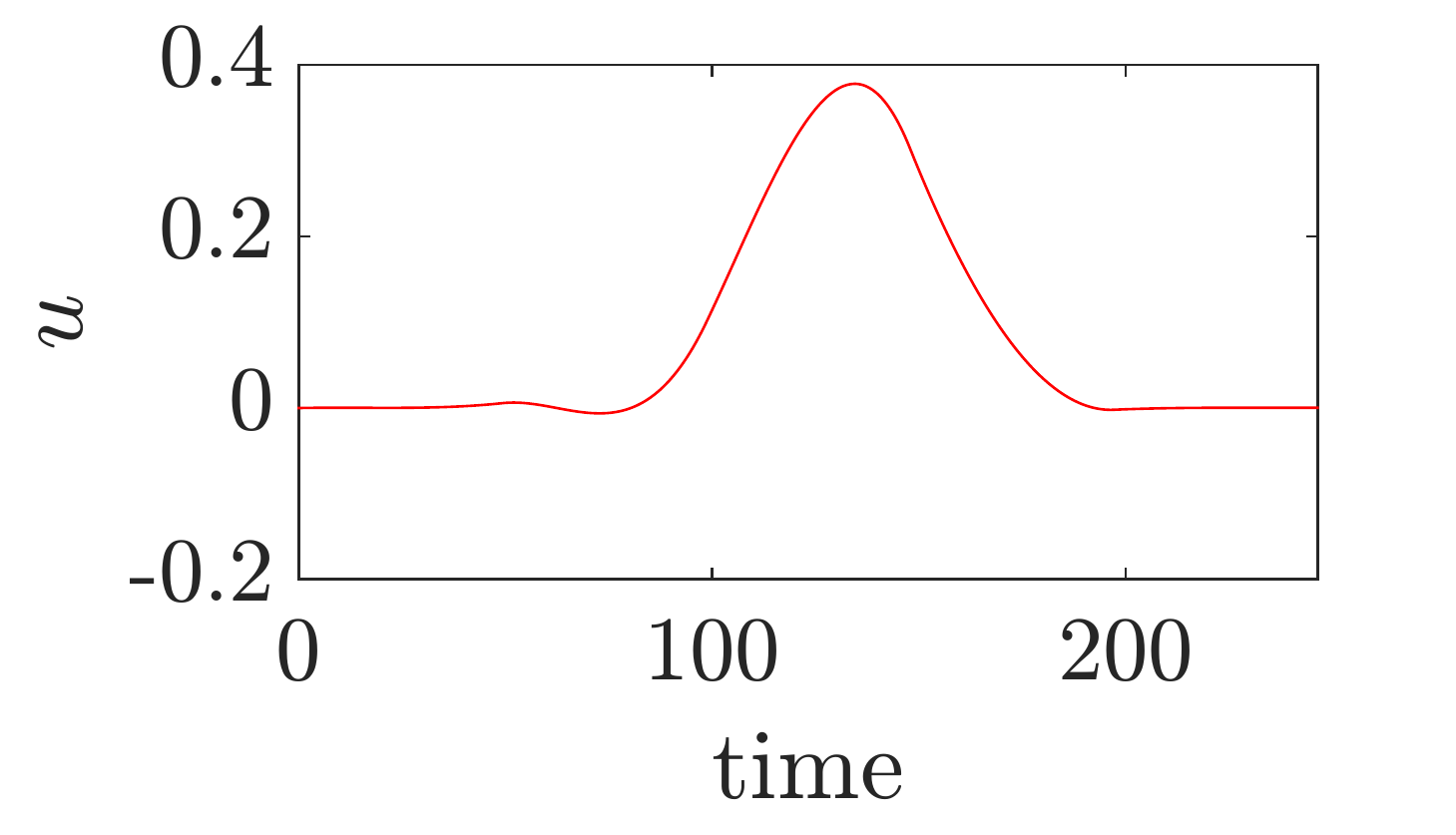}}};
		\node[anchor=north west] (kappa.center)  at (0.90,0)
		{{\includegraphics[width=0.202\textwidth, clip=true,trim=0.1in 0.1in 0.55in 0.13in]{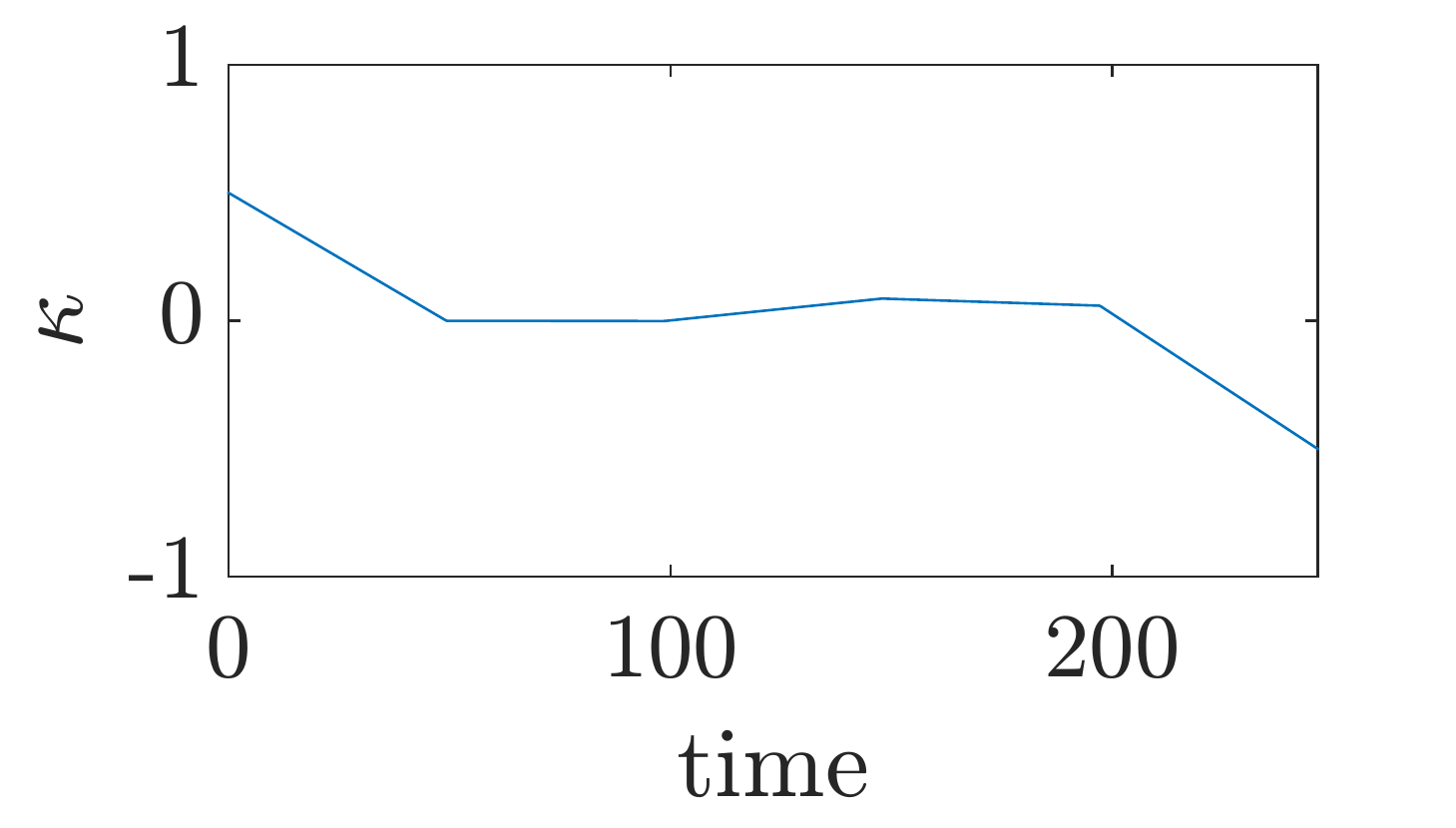}}};
\end{tikzpicture}\label{fig:TRO_figure_bend_control}}		
	\caption{Shortest path with curvature cost example for bendable cuboid in $SE(3)$}
\label{fig:bend_cuboid_optimal}
\end{figure*}


This section solves several optimal path planning problems proposed in \S\ref{sec:prob} for a rigid and soft cuboid robot using the $L_p$ constraints proposed in the previous section. The equations of motion are those from (\ref{eqn:kinematicsSE3}-\ref{eqn:kinematicsSE32}) where the rotational matrix is represented by the unit quaternion \cite{kuipers1999quaternions, jia2008quaternions}. Assume that the control inputs, $\roboControlLVel, \roboControlAVel_1,\roboControlAVel_2,\roboControlAVel_3 \in C([0,T_f],\mathbb{R})$, are bounded by $\roboControlLVel(t)\in [-30, 30]$ and $\roboControlAVel_i(t)\in [-\pi/2, \pi/2]$ for each $i$. In addition, the curvature, $\roboCurvature\in C([0,T_f],\mathbb{R})$, is a time-varying control variable bounded by $[-1, 1]$ for the bendable cuboid case.

The MATLAB-based numerical optimal control solver OPTRAGEN
\cite{Bhattacharya2006}, converts the Bolza type optimal control problem
into a nonlinear programming problem (NLP) by approximating the trajectories
using B-splines, and sample at each provided collocation points. After
converting the optimal control problem to a NLP, invoking a numerical NLP
solver using an interior point method, IPOPT 3.12.6 \cite{Wachter2006}, find
the numerical solution.

The rectangular obstacle with half length $\halfL^1=(10,2,5)$ is located at the origin of the world frame, $(0,0,0)$, and it is rotated by $\pi/4$ along the axis of $(1, 1, 0)$. In the rigid regular cuboid model, the cost function is chosen, 
$L(\rPose, \dot{\rPose}, u):=\sqrt{\dot{x}^2+\dot{y}^2+\dot{z}^2}$, to give a shortest path
problem. The half length of the rigid regular cuboid robot is given by $(2,1,1)$, and the CoM of robot is initially located at $(-4, -4, -4)$. The robot is initially rotated by $\pi/4$ along the axis of $(1,0,0)$. The desired final CoM position is $(4,4,4)$ with the rotation of $\pi/4$ along the axis of $(0,0,1)$. 

In the bendable cuboid model, a regularization term is added to the cost to minimize the bending effort assuming that the cost of bending is proportional to $|\roboCurvature|$, 
\begin{equation*}
L(\rPose, \dot{\rPose}, u):=\sqrt{\dot{x}^2+\dot{y}^2+\dot{z}^2} + \lambda|\roboCurvature|.
\end{equation*}
The regularization parameter $\lambda = 0.1$. The half length of the bendable cuboid robot is given by $(5,1,1)$, and the initial and final curvature is given as $\roboCurvature(0)=0.5$ and $\roboCurvature(T_f)=-0.5$, respectively. 

For all of the following solution figures in Figure~\ref{fig:cuboid_optimal}, the red surface represents the obstacle boundaries, the blue surface represents the robot boundaries at initial configuration, and the green surface represents the robot boundary at final configuration. The black surfaces are used to demonstrate the approximated shape of the robot at each sampled collocation time point. The solid blue curve represents the trajectory of CoM.

\subsection{Rigid model}
The numerical result is shown in Figure~\ref{fig:TRO_figure_cub_cub_Lp}-\ref{fig:TRO_figure_cub_cub_Nec_Suff}. It is confirmed numerically at each collocation time that the closest point has been correctly captured, and the edge-to-edge collision is avoided by the proposed safety constraint. In addition, the necessary and sufficient condition for the closest point is also verified in Figure~\ref{fig:TRO_figure_cub_cub_Nec_Suff}, where the blue and red arrow represents the outward normal of the robot and the extended level set of $L_p$ at the closest point, respectively. In addition, the green arrows shows the basis vectors of the tangent space at the closest point. The corresponding numerical solution of four controls are shown in Figure~\ref{fig:TRO_figure_cuboid_control}. The result suggests that the robot should first focus on rotating without moving the CoM excessively, and then once it points toward the right direction, it navigates close to the final position. After reaching near by the final position, the robot should focus back on rotating the robot to meet the final orientation constraint. The result is intuitively understandable since the cost was to minimize the traveled length of the CoM.

\subsection{Bendable model}

In addition to the kinematic controls, $\roboControlLVel$ and $\roboControlAVel$, a continuous curvature which connects initial and final curvature is sought, and the result are shown in Figure~\ref{fig:TRO_figure_bend_cub_Lp}-\ref{fig:TRO_figure_bend_cub_Nec_Suff}. Similar to the rigid model, the orthogonality condition, and the opposite sign condition for normal vectors are confirmed. The continuous change in the curvature deforms the cuboid robot in Figure~\ref{fig:TRO_figure_bend_cub_Lp}. Since the cost consists of the traveled length of the CoM, and the integral of the absolute curvature, $|\roboCurvature|$, the numerical solution finds a path where the trajectory of the CoM turns close to the obstacle, and flattens out the robot to maintain the low curvature during the navigation. The simulation numerically and graphically verifies that the computed path avoids the collision. The corresponding numerical solution of five controls are shown in Figure~\ref{fig:TRO_figure_bend_control}. 
Similar to the result in the cuboid case, the result suggests that the
robot should turn first, and move towards to the final position, and
then rotate back to meet the final orientation constraint. In addition,
in order to minimize the norm of the curvature, the numerical planner
first decreases the curvature to near zero and stays close to zero, then
changes to the final negative curvature at the end.

%
%

\section{Conclusions}
\label{sec:conclusion}


A new framework of approximating the regular cuboid and bent cuboid robot
has been introduced and used to provide analytic safety constraints for
collision avoidance constraints associated to cuboid shaped obstacles. 
The weighted $L_p$ norm approximately models the surfaces of regular
cuboids, and a new positive definite value function derived from a weighted
$L_p$ function in polar coordinates approximately models the surfaces of 
bent cuboids. Collision avoidance requires the surface level sets between
the robot and the obstacles to be disjoint, or equivalently, to be such that
the nearest point from the robot to the obstacle lies outside of the
obstacle level set.  The optimal control planning problem gets converted to 
a two stage optimization problem, where the inner optimization stage finds
the potential closest point in the weighted $L_p$ metric of the obstacle. 
were explored and used to reduce the two stage optimization to a single
optimization problem. Simulation results graphically show the usefulness
and viability of the proposed algorithm.

%
%


%

\appendices
%
%

\section{Proof of Theorem~\ref{thm:polar_Lp_2D}}
\label{app:polar}

Let $\roboRadiusCurvature=1/|\roboCurvature|$ be the radius of curvature. The regular rectangle approximated by the $(\halfL_1,\halfL_2)$-weighted $L_p$ norm is
\begin{equation}
  B((\halfL_1,\halfL_2),p)
    := \{u\in\mathbb{R}^2|\ ||u||_{(\halfL_1,\halfL_2),p}\leq 1\}
\end{equation} 
where $u=(u_x,u_y)$ such that $u_x\in[-\halfL_1,\halfL_1]$ and $u_y\in[-\halfL_2,\halfL_2]$. 

Now define the mapping $\mathcal{B}_\roboCurvature: B((\halfL_1,\halfL_2),p)\to\mathbb{R}^2$, 
\begin{equation} \label{eqn:warping_x}
  \mathcal{B}_\roboCurvature(u)
    := \begin{pmatrix}
        (\roboRadiusCurvature+u_y)\cos(\alpha\pi/2+(\theta_\roboCurvature/2)(u_x/\halfL_1))
\\
        (\roboRadiusCurvature+u_y)\sin(\alpha\pi/2+(\theta_\roboCurvature/2)(u_x/\halfL_1))-\alpha\roboRadiusCurvature
      \end{pmatrix}
\end{equation}
for all $u\in B((\halfL_1,\halfL_2),p)$, where $\alpha=\sign(\roboCurvature)$. 

Observe that $\mathcal{B}_\roboCurvature(\cdot)$ is invertible on its image,
with said inverse $\mathcal{B}^{-1}_\roboCurvature: \mathcal{B}_\roboCurvature(B((\halfL_1,\halfL_2),p))\to B((\halfL_1,\halfL_2),p)$ given by
\begin{equation}
  \begin{pmatrix}u_x\\ u_y\end{pmatrix} 
    =
  \begin{pmatrix}
    2(\arctan{((\overline{u}_y+\alpha\roboRadiusCurvature)/\overline{u}_x)}-\alpha\pi/2)\halfL_1/\theta_\roboCurvature
\\
    \sqrt{\overline{u}_x^2+(\overline{u}_y+\alpha\roboRadiusCurvature)^2}-\roboRadiusCurvature
  \end{pmatrix}
\end{equation}
where $(\overline{u}_x,\overline{u}_y)\in \mathcal{B}_\roboCurvature(B((\halfL_1,\halfL_2),p))$, and $\theta_\roboCurvature$ is the bending angle of constant curvature (see Figure~\ref{fig:Lp_Bent_levelset_3d_positive}). By using Assumption~\ref{asmp:invariant_center} and (\ref{eqn:bending_angle}), with the coordinate transformation in (\ref{eqn:curvature_isomorphism}), it holds that 
\begin{equation}
  \begin{pmatrix}|\roboCurvature|u_x\\ |\roboCurvature|u_y\end{pmatrix}
  =
  \begin{pmatrix}
    (\theta_{\mathcal{T}_\roboCurvature}(\overline{u})-\alpha\pi/2) \\
    R_{\mathcal{T}_\roboCurvature}(\overline{u})-1
  \end{pmatrix}.
\end{equation}
Since $u\in B((\halfL_1,\halfL_2),p)$, $|||\roboCurvature|u||_{(\halfL_1,\halfL_2),p}\leq |\roboCurvature|$ holds. Therefore, by computing $|||\roboCurvature|u||_{(\halfL_1,\halfL_2),p}$, 
\begin{equation}
  \Phi_{(\halfL, \roboCurvature, p)}(\overline{u}) \leq |\roboCurvature|,
\end{equation}
for all $\overline{u}\in \mathcal{B}_\roboCurvature(B((\halfL_1,\halfL_2),p))$. 
Let the approximated bent rectangular be 
\begin{equation}
\label{eqn:boundary_bent_Lp}
B((\halfL_1,\halfL_2),\roboCurvature,p):=\{\overline{u}\in\mathbb{R}^2| \Phi_{(\halfL, \roboCurvature, p)}(\overline{u})\leq |\roboCurvature|\}.
\end{equation}
Since the mapping, $\mathcal{B}_\roboCurvature$, is diffeomorphic, $B((\halfL_1,\halfL_2),\roboCurvature,p)$ and $B((\halfL_1,\halfL_2),p)$ have a one to one correspondence for all $p$. Furthermore, the boundary of $ B((\halfL_1,\halfL_2),\roboCurvature,p)$ and the boundary of $B((\halfL_1,\halfL_2),p)$ can be induced by each other as well. Therefore, it is now obvious that $\partial B((\halfL_1,\halfL_2),\roboCurvature,p)$ approaches the bent rectangular robot as $p \rightarrow \infty$ by computing (\ref{eqn:warping_x}) for the actual bent rectangular robot. 

\textbf{Remark.}
The centerline of $B((\halfL_1,\halfL_2),p)$ is on the $x$-axis, and is given by a parameterized curve $C(s)=(s,0)$ for $s \in [-\halfL_1,\halfL_1]$. Transforming the curve $C(\cdot)$ using $\mathcal{B}_\roboCurvature$, leads to an arc passing through zero with constant curvature $\roboCurvature$, 
\begin{equation} \label{eqn:bent_centerline}
  \begin{pmatrix}
  (\roboRadiusCurvature)\cos(\alpha\pi/2+(\theta_\roboCurvature/2)(s/\halfL_1))
  \\
  (\roboRadiusCurvature)\sin(\alpha\pi/2+(\theta_\roboCurvature/2)(s/\halfL_1))
    -\alpha\roboRadiusCurvature
  \end{pmatrix}.
\end{equation}
This shows that the $x$-axis has been deformed to an arc after the transformation $\mathcal{B}_\roboCurvature$, and that the new bent axis corresponds to $\mathcal{B}_\roboCurvature(C(\cdot))$, which serves as a centerline of the bent robot. 
Two examples of bent centerlines are shown as black curves in Figure~\ref{fig:Lp_Bent_levelset_positive_negative} for positive and negative curvature, respectively. 


\section{Proof of Theorem~\ref{thm:necessary_regular}}
\label{app:necessary_regular}
Let $H(\vPose_x,\vPose_y,\vPose_z,\lambda)$ be a Hamiltonian of the problem in (\ref{eqn:two_stage_opt}), 
\begin{equation}
\label{eqn:Hamiltonian}
H(\vPose_x,\vPose_y,\vPose_z,\lambda):= \left| \left| w(\vPose)\right| \right|_{\halfL^j,\oP}+\lambda(\left| \left| {\vPose}\right| \right|_{\halfL,\rP}-1),
\end{equation}
where $w(v):=(w_x,w_y,w_z)$. The gradient of $H$ is 
\begin{equation} \label{eqn:Hamiltonian_grad}
  \nabla_{\vPose} H = 
    \left| \left| w(\vPose)\right| \right|_{\halfL^j,\oP}^{{1-\oP}}([{\rPose_t}]^{-1}\circ \sPose^j({\mathcal{W}}))^T +\lambda\left| \left| {\vPose}\right| \right|_{\halfL,\rP}^{{1-\rP}}(\mathcal{V}_{\overline{\vPose}})^T.
\end{equation}
Then the necessary condition for a stationary point is 
\begin{equation} \label{eqn:Hamiltonian_grad0}
  \nabla_{\vPose} H = 0
\end{equation}
with the equality constraint in (\ref{eqn:two_stage_equality_regular}). By multiplying $\overline{\vPose}$ to the right and using (\ref{eqn:two_stage_equality_regular}), $\lambda$ is computed to be
\begin{equation*}
  \lambda = -(\left| \left| w(\vPose)\right| \right|_{\halfL^j,\oP}^{{1-\oP}}/\left| \left| {\vPose}\right| \right|_{\halfL,\rP}^{{1-\rP}})([{\rPose_t}]^{-1}\circ \sPose^j({\mathcal{W}}))^T\overline{\vPose},
\end{equation*}
and by replacing $\lambda$ in (\ref{eqn:Hamiltonian_grad}), the necessary condition in (\ref{eqn:Hamiltonian_grad0}) is equivalent to (\ref{eqn:necessary_regular}). 

\section{Proof of Theorem~\ref{thm:necessary_bent}}
\label{app:necessary_bent}
Let $H(\vPose_x,\vPose_y,\vPose_z,\lambda)$ be a Hamiltonian of the problem in (\ref{eqn:two_stage_opt}), 
\begin{equation}
\label{eqn:Hamiltonian_bent}
H(\vPose_x,\vPose_y,\vPose_z,\lambda):= \left| \left| w(\vPose)\right| \right|_{\halfL^j,\oP}+\lambda(\Psi_{(\halfL,\roboCurvature,\rP)}(\vPose)-|\roboCurvature|),
\end{equation}
where $w(v):=(w_x,w_y,w_z)$. The gradient of $H$ is 
\begin{equation} \label{eqn:Hamiltonian_grad_bent}
  \left| \left| w(\vPose)\right| \right|_{\halfL^j,\oP}^{{1-\oP}}([{\rPose_t}]^{-1}\circ \sPose^j({\mathcal{W}}))^T +\lambda|\roboCurvature|\Psi_{(\halfL,\roboCurvature,\rP)}(\vPose)^{1-\rP}(\mathcal{U}_{\overline{\vPose}}).
\end{equation}
Then the necessary condition for a stationary point is 
\begin{equation} \label{eqn:Hamiltonian_grad0_bent}
  \nabla_{\vPose} H = 0
\end{equation}
with the equality constraint in (\ref{eqn:two_stage_equality_regular}). By multiplying $\overline{Q}\nu$ to the right and using (\ref{eqn:two_stage_equality_regular}) and (\ref{eqn:equality_psi}), $\lambda$ is computed to be
\begin{multline} \label{eqn:lambda_bent}
  \lambda = 
    -(\left| \left| w(\vPose)\right| \right|_{\halfL^j,\oP}^{{1-\oP}}\Psi_{(\halfL,\roboCurvature,\rP)}(\vPose)^{\rP-1}/|\roboCurvature|^{p_r+1})
    \\
    ([{\rPose_t}]^{-1}\circ \sPose^j({\mathcal{W}}))^T\overline{Q}\nu,
\end{multline}
and by replacing $\lambda$ in (\ref{eqn:Hamiltonian_grad_bent}), the necessary condition in (\ref{eqn:Hamiltonian_grad0_bent}) is equal to (\ref{eqn:necessary_bent}).

%
%
%

\ifCLASSOPTIONcaptionsoff
  \newpage
\fi



\bibliographystyle{IEEEtran}
\bibliography{IEEETRO_2017_Nak_seung_Patrick_Hyun}
\end{document}